\documentclass[leqno,onefignum,onetabnum]{siamltex1213}
\pdfoutput=1
\usepackage{amsfonts,amsmath,amssymb,mathrsfs}
\usepackage[commentmarkup=uwave]{changes}
\usepackage{booktabs,ctable,threeparttable,multirow}
\usepackage{graphicx,boxedminipage,appendix}
\usepackage[caption=false]{subfig}
\usepackage{algorithm,algorithmic}
\usepackage{lineno,soul,color,xcolor}

\usepackage{amsthm,bm,enumerate,hyperref}
\newtheorem{theoremmy}{Theorem}[section]

\newtheorem{exper}[theoremmy]{Experiment}
\newtheorem{problem}[theoremmy]{Problem}
\numberwithin{equation}{section}

\newcommand{\new}[0]{\mathrm{new}}

\newcommand{\spans}[0]{\mathrm{span}}

\newcommand{\N}[0]{\mathcal{N}}
\newcommand{\X}[0]{\mathcal{X}}

\newcommand{\Z}[0]{\mathcal{Z}}
\newcommand{\UU}[0]{\mathcal{U}}
\newcommand{\V}[0]{\mathcal{V}}

\newcommand{\bsmallmatrix}[1]{\begin{bmatrix}\begin{smallmatrix}
		#1\end{smallmatrix}\end{bmatrix}}

\title{Refined and refined harmonic Jacobi--Davidson methods for
computing several GSVD components of a large regular matrix pair\thanks{The
work of the first author was supported by the Youth Fund of
the National Science Foundation of China (No. 12301485) and the Youth Program
of the Natural Science Foundation of Jiangsu Province (No. BK20220482), and
the work of the second author was supported by the National Science Foundation
of China (No. 12171273).}}

\author{
Jinzhi Huang\thanks{School of Mathematical Sciences, Soochow University,
	215006 Suzhou, China
	(\url{jzhuang21@suda.edu.cn})}
\and
Zhongxiao Jia\thanks{Corresponding author. Department of Mathematical Sciences,
	Tsinghua University, 100084 Beijing, China
	(\url{jiazx@tsinghua.edu.cn}).}}
\begin{document}
\maketitle

\begin{abstract}
Three refined and refined harmonic extraction-based
Jacobi--Davidson (JD) type methods are proposed, and their
thick-restart algorithms with deflation and purgation are
developed to compute several generalized singular value
decomposition (GSVD) components of a large regular matrix pair.
The new methods are called refined cross product-free (RCPF),
refined cross product-free harmonic (RCPF-harmonic) and refined
inverse-free harmonic (RIF-harmonic) JDGSVD algorithms,
abbreviated as RCPF-JDGSVD, RCPF-HJDGSVD and RIF-HJDGSVD, respectively.
The new JDGSVD methods are more efficient than the corresponding
standard and harmonic extraction-based JDSVD methods proposed previously
by the authors, and can overcome the erratic behavior
and intrinsic possible non-convergence of the latter ones.
Numerical experiments illustrate
that RCPF-JDGSVD performs better for the computation of
extreme GSVD components while RCPF-HJDGSVD and RIF-HJDGSVD suit better for
that of interior GSVD components.
\end{abstract}



\pagestyle{myheadings}
\thispagestyle{plain}
\markboth{REFINED AND REFINED HARMONIC JDGSVD METHODS}
{JINZHI HUANG AND ZHONGXIAO JIA}

\section{Introduction}\label{sec:1}

The GSVD was initially established by Van Loan
\cite{van1976generalizing} and developed by Paige and
Saunders \cite{paige1981towards}, and it has soon become one of the
most important matrix decompositions
\cite{golub2012matrix,stewart2001matrix,stewart90}.

Let $A\in\mathbb{R}^{m\times n}$ and
$B\in\mathbb{R}^{p\times n}$ with $m+p\geq n$.
Suppose that $\N(A)\cap\N(B)=\{\bm{0}\}$,
where $\N(\cdot)$ denotes the null space of a matrix.
Then $(A,B)$ is called a regular matrix pair.
Write $q_1=\dim(\N(A))$, $q_2=\dim(\N(B))$ and $l_1=\dim(\N(A^T))$, $l_2=\dim(\N(B^T))$, respectively, where the superscript $T$ denotes
the transpose of a matrix and $\dim(\cdot)$ denotes the dimension of a subspace.
Then the GSVD of $(A,B)$ is as follows:
\begin{equation}\label{Gsvd}
	\left\{\begin{aligned}
		&A=U\Sigma_AX^{-1}, \\[0.3em]
		&B=V\Sigma_BX^{-1},
	\end{aligned}\right.
\qquad\mbox{with}\qquad
	\left\{\begin{aligned}
		&\Sigma_A=\diag\{C,\mathbf{0}_{l_1, q_1},I_{q_2}\},
		 \\[0.3em]
		&\Sigma_B=\diag\{S,I_{q_1},\mathbf{0}_{l_2, q_2}\},
	\end{aligned}\right.
\end{equation}
where $U=[U_q,U_{l_1},U_{q_2}]$ and $V=[V_q,V_{q_1},V_{l_2}]$
are orthogonal, $X=[X_q,X_{q_1},X_{q_2}]$ is nonsingular,
and the diagonal
$C=\diag\{\alpha_1,\dots,\alpha_q\}$ and
$S=\diag\{\beta_1,\dots,\beta_q\}$ satisfy
\begin{equation*}
	0<\alpha_i,\beta_i<1 \qquad\mbox{and}\qquad
	\alpha_i^2+\beta_i^2=1,  \qquad i=1,\dots,q
\end{equation*}
with $q=n-q_1-q_2$.
Here the subscripts in the block submatrices of $U,V,X$ are
their column numbers, and $I_{i}$ and $\bm{0}_{i,j}$ denote
the $i$-by-$i$ identity matrix and $i$-by-$j$ zero matrix,
respectively. The subscripts are dropped whenever their sizes are clear from
the context. Let $u_i$, $v_i$ and $x_i$ be the $i$th columns of $U_q$, $V_q$
and $X_q$, respectively, $i=1,\dots,q$.
Then the quintuples $(\alpha_i,\beta_i,u_i,v_i,x_i)$,
$i=1,\dots,q$ are called {\em nontrivial} GSVD components of $(A,B)$.
In particular, the scalar pairs $(\alpha_i,\beta_i)$ or,
equivalently, $\sigma_i=\frac{\alpha_i}{\beta_i}$ are called
the nontrivial generalized singular values of $(A,B)$,
and $u_i,v_i$ and $x_i$ are the corresponding left and right
generalized singular vectors, respectively.

From \eqref{Gsvd}, for $i=1,2,\ldots,q$, the GSVD of $(A,B)$ can be
written in the form
\begin{equation}\label{gsvdvector}
	\left\{\begin{aligned}
		Ax_i&=\alpha_i u_i,   \\[0.3em]
		Bx_i&=\beta_i v_i, \\[0.3em]
		\beta_i A^Tu_i&=\alpha_i B^Tv_i.
	\end{aligned}\right.
\end{equation}
Denote by $(\alpha_i,\beta_i)=(0,1)$ or $(\alpha_i,\beta_i)=(1,0)$
a trivial zero or infinite generalized singular value.
Then the above form still holds with $u_i,v_i$ and $x_i$
being the left and right generalized singular vectors corresponding
to the zero or infinite generalized singular value.
From \eqref{Gsvd}, we have $X^T(A^TA+B^TB)X=I_n$.
Therefore, $X$ is $(A^TA+B^TB)$-orthogonal, and its columns
$x_i$'s are of $(A^TA+B^TB)$-norm unit length.
Naturally, we require that any approximation to $x_i$ have the
same length.

In this paper, we consider the following GSVD computational problem of
a large and possibly sparse regular matrix pair $(A,B)$.

\begin{problem}\label{probl}
For a given target $\tau>0$, label all the nontrivial
generalized singular values of $(A,B)$ as
\begin{equation}\label{order}
	|\sigma_1-\tau|\leq\dots\leq|\sigma_{\ell}-\tau|<
	|\sigma_{\ell+1}-\tau|\leq\dots\leq|\sigma_q-\tau|.
\end{equation}
We want to compute the GSVD components
$(\alpha_i,\beta_i,u_i,v_i,x_i)$, $i=1,\dots,\ell$ associated
with the $\ell$ generalized singular values
$\sigma_i$, $i=1,\dots,\ell$ of $(A,B)$ closest to $\tau$.
\end{problem}

If $\tau$ is inside the spectrum of the nontrivial generalized
singular values of $(A,B)$, then those $(\alpha_i, \beta_i, u_i,
v_i, x_i)$'s are called interior GSVD components of $(A,B)$;
if $\tau$ is close to one of the ends of the nontrivial generalized singular
spectrum, then they are called the extreme, i.e., largest or smallest, ones.
In the sequel, we assume that the target $\tau$ is
not equal to any generalized singular value of $(A,B)$.

Hochstenbach \cite{hochstenbach2009jacobi} proposes a
Jacobi--Davidson (JD) type GSVD method, called JDGSVD,
to compute several extreme or interior GSVD components of
$(A,B)$ where $B$ has full column rank.
At the subspace expansion phase, an
$(m+n)$-by-$(m+n)$ linear system, called the correction equation,
needs to be solved iteratively;
for analysis and details on the accuracy requirement on the
inner iterations of JD type methods for eigenvalue
and SVD problems, see \cite{huang2019inner,huang2022harmonic,
huang2023cross,jia2014inner,jia2015harmonic}, where it is shown
that it generally suffices to solve
correction equations with \emph{low or modest}
accuracy; that is, the relative errors of approximate solutions
lie in $[10^{-4},10^{-2}]$. More generally, the JDGSVD method formulates the
GSVD of $(A,B)$ either as the generalized eigendecomposition of the
augmented matrix pair $\left(\bsmallmatrix{&A\\A^T&},\bsmallmatrix{I&\\&B^TB}\right)$
for $B$ of full column rank or that of
$\left(\bsmallmatrix{&B\\B^T&},\bsmallmatrix{I&\\&A^TA}\right)$
for $A$ of full column rank, computes the corresponding
generalized eigenpairs, and then reconstructs the desired
approximate GSVD components from the relevant converged
eigenpairs. For the second formulation,  an
$(p+n)$-by-$(p+n)$ correction equation is solved iteratively
at each subspace expansion step. However, as has been theoretically
proven and numerically confirmed in
\cite{huang2021choices,huang2023cross},
a fairly ill conditioned $B$ or $A$ may make
the corresponding JDGSVD method numerically backward
unstable for the GSVD problem itself,
even if the relative residuals of approximate
eigenpairs of the underlying generalized
eigenvalue problem are already at the level of machine precision.

Zwaan and Hochstenbach \cite{zwaan2017generalized} present
two GSVD methods, called the generalized Davidson (GDGSVD)
and multidirectional (MDGSVD) methods, to compute several
extreme GSVD components of $(A,B)$.
The right searching subspace for the GDGSVD method is spanned
by the residuals of the generalized Davidson method
\cite[Sec.~11.2.4 and Sec.~11.3.6]{bai2000} applied to the
eigenvalue problem of the cross-product matrix pair
$(A^TA,B^TB)$; that for the MDGSVD method is first
expanded by two dimensions with the vectors formed by
premultiplying the best approximate right generalized
singular vector with $A^TA$ and $B^TB$,
and then truncated by one dimension so that an inferior
search direction is discarded.
The left searching subspaces for these two methods are
formed by premultiplying the right one with $A$ and $B$, respectively.
These two methods make use of the standard extraction approach
to compute the approximate GSVD components without
explicitly forming $A^TA$, $B^TB$.
Zwaan \cite{zwaan2019} utilizes the Kronecker canonical
form of a matrix pair \cite{stewart90}, and proves
that the GSVD of $(A,B)$ is equivalent to the generalized
eigendecomposition of a matrix pair with much larger order $2m+p+n$.
Though the pair does not involve cross-products or any
other matrix-matrix product, this formulation may be mainly
of theoretical value since (i) the nontrivial generalized
eigenvalues and eigenvectors of the larger structured matrix
pair come in quadruples and are always complex, (ii)
the conditioning of the structured generalized eigenvalue problem is unknown,
and (iii) it is extremely hard to propose a numerically backward stable
structure-preserving algorithm.

Adapted standard and harmonic Rayleigh--Ritz projections
\cite{golub2012matrix,parlett1998symmetric,saad2011,stewart2001matrix},
or called the standard and harmonic extraction approaches, for eigenvalue and
generalized eigenvalue problems to GSVD problems, the authors have
recently proposed the cross product-free (CPF) JDGSVD \cite{huang2023cross},
the CPF-harmonic and
inverse-free (IF) harmonic JDGSVD methods \cite{huang2022harmonic}, written as
CPF-JDGSVD, CPF-HJDGSVD and IF-HJDGSVD for short, to solve Problem~\ref{probl}.
At the subspace expansion step, each of these methods requires
an approximate solution of its own $n$-by-$n$ correction
equation, and uses it to expand the right searching subspace;
the two left searching subspaces are formed by
premultiplying the right one with $A$ and $B$, respectively.
The three methods are fundamentally different in the extraction phase.
The CPF-JDGSVD method works on $(A,B)$ directly, applies
the standard extraction approach to the left and right
searching subspaces, and computes the GSVD of a small
projection matrix pair. It implicitly realizes the
standard Rayleigh--Ritz projection of the generalized
eigenvalue problem of $(A^TA,B^TB)$ onto the right
searching subspace \cite{huang2023cross}.
With $B$ of full column rank, the CPF-HJDGSVD method
implicitly realizes the harmonic extraction approach
of the singular value decomposition (SVD)  problem of
$AL^{-T}$ onto the right and one of the left searching
subspaces, where $L^{T}\in\mathbb{R}^{n\times n}$ is the
factor in the sparse Cholesky factorization $B^TB=LL^T$.
At each extraction step, the method needs to solve the generalized
eigenvalue problem of a small symmetric positive definite
matrix pair. For a general and possibly rank deficient $B$, the
IF-HJDGSVD method implicitly carries out the harmonic
extraction of the generalized eigenvalue problem of
$(A^TA,B^TB)$ onto the right searching subspace, and
computes the generalized eigendecomposition of a small
symmetric positive definite matrix pair; it is the
inverses $(A^TA)^{-1}$ and $(B^TB)^{-1}$-free, and
works for a general matrix pair $(A,B)$.
For justifications and details, we refer the reader
to \cite{huang2022harmonic}

Just like those standard Rayleigh--Ritz methods for the matrix eigenvalue
problem and the SVD problem
\cite{cullum2002lanczos,hochstenbach2001jacobi,hochstenbach2009jacobi,simon2000lowrank,stoll2012krylov,zwaan2017generalized},
CPF-JDGSVD suits better for the computation of extreme generalized
singular values of $(A,B)$.
However, we deduce from, e.g., \cite{jia2004some, jia2003implicitly,jiastewart2001},
that the approximate generalized singular vectors obtained by
it may converge erratically or even fail to converge
even if the approximate generalized singular values converge.
These phenomena have been numerically observed in \cite{huang2023cross}.

The refined extraction approach or refined Rayleigh--Ritz
projection was initially proposed by the second author in
\cite{jia1997refined}, and has been intensively studied and developed
in, e.g., \cite{hochstenbach2004harmonic,
hochstenbach2008harmonic,jia1999,jia2005,
jia2015harmonic,jia2003implicitly,jia2010refined,jiastewart2001,
kokiopoulou2004computing,wu16primme,wu2015preconditioned}
for the eigenvalue and SVD problems.
For the large matrix eigenvalue problem, it
is systematically accounted for in the books
\cite{bai2000,stewart2001matrix,vandervorst}.
As is shown, the refined extraction has better convergence,
and fixes the erratic convergence
behavior and possible non-convergence of standard and harmonic
extractions \cite{jia1997refined,jia1999,jia2004some,jia2005,
jia2010refined,jiastewart2001}; also see, e.g., \cite{hochstenbach2004harmonic,
hochstenbach2008harmonic,kokiopoulou2004computing,wu16primme,wu2015preconditioned}.
Importantly, the basic convergence results in
\cite{jia2004some,jia2005,jiastewart2001} adapted to CPF-JDGSVD,
CPF-HJDGSVD and IF-HJDGSVD indicate that these three methods
inherit those mentioned convergence deficiencies of the standard and
harmonic extractions; that is, the three methods may work erratically
and inefficiently.

In this paper, in order to fix the deficiency of CPF-JDGSVD,
CPF-HJDGSVD and IF-HJDGSVD and to better solve Problem~\ref{probl},
we will propose three refined extraction-based JDGSVD methods.
We first present a refined JDGSVD (RCPF-JDGSVD) method. It
computes the approximate generalized singular values by the
standard extraction-based JDGSVD method but nontrivially adapts
the refined extraction of the generalized eigenvalue problem of
$(A^TA,B^TB)$ to the GSVD problem of $(A,B)$, and
seeks new approximate generalized
singular vectors that are generally more and can be much more
accurate than those obtained by the standard extraction.

Like the interior eigenvalue and SVD problems,
for the computation of interior GSVD components, the standard
extraction may produce spurious Ritz values
and has difficulty to pick up good Ritz
values, if any, correctly even if the searching subspaces are
sufficiently good
\cite{jia2015harmonic,jiastewart2001,stewart2001matrix,vandervorst},
causing that the CPF-JDGSVD method may converge slowly or even
fail to converge, as has been numerically confirmed in
\cite{huang2022harmonic}.
Whenever Ritz values are poor or, though good,
they are selected incorrectly, a refined extraction-based method
certainly delivers incorrect approximate GSVD components,
which severely affects the correct expansion of
underlying subspaces in JDGSVD type methods.
As a result, the refined extraction-based method may perform
poorly when computing interior GSVD components.
Just as the harmonic extraction-based methods for the eigenvalue
and SVD problems \cite{stewart2001matrix,vandervorst} that suit
better for computing interior eigenpairs and singular triplets
\cite{hochstenbach2004harmonic,hochstenbach2008harmonic,
	huang2019inner,jia2002refinedh,jia2015harmonic,jia2010refined,
	morgan1998harmonic,morgan2006harmonic},
our two harmonic extraction-based CPF-HJDGSVD and IF-HJDGSVD
methods are more suitable for computing interior generalized
singular values. Nevertheless, the harmonic Ritz vectors may have erratic
convergence behavior and may even fail to converge
because of spurious harmonic Ritz value(s) even if
searching subspaces are sufficiently accurate \cite{jia2005,jia2015harmonic}.

In order to overcome the aforementioned deficiency,
for the computation of interior GSVD components, on the basis of
CPF-HJDGSVD and IF-HJDGSVD, we propose refined harmonic
JDGSVD type methods that retain the merits of the harmonic
JDGSVD methods for computing generalized singular values
but seek more accurate approximate generalized singular vectors by
using the refined extraction in a proper way.
The resulting methods are abbreviated as RCPF-HJDGSVD
and RIF-HJDGSVD, respectively.

We first focus on the case $\ell=1$, and propose
basic refined and refined harmonic extraction-based
JDGSVD methods for the GSVD problem of interest.
Then combining the methods with appropriate restart,
deflation and purgation, we develop thick-restart RCPF-JDGSVD,
RCPF-HJDGSVD and RIF-HJDGSVD algorithms for Problem~\ref{probl}
with $\ell>1$, with details on
effective and efficient implementations described.
We will numerically demonstrate that they have better convergence
behavior and are considerably more efficient than the corresponding
standard and harmonic extraction-based JDGSVD algorithms.
We also illustrate that RCPF-JDGSVD performs better than
RCPF-HJDGSVD and RIF-HJDGSVD for extreme GSVD components,
but for interior GSVD problems the latter two ones are
preferable and RIF-HJDGSVD has wider applicability than RCPF-HJDGSVD.

The rest of this paper is organized as follows.
In Section~\ref{sec:2}, we review
the CPF-JDGSVD method in \cite{huang2023cross}
and the CPF-HJDGSVD and IF-HJDGSVD methods in \cite{huang2022harmonic}.
In Section~\ref{sec:3}, we introduce a refined extraction approach
for the GSVD computation; combining it with the standard,
CPF-harmonic and IF-harmonic extractions, we propose the refined,
refined CPF-harmonic and refined IF-harmonic extraction-based
JDGSVD methods: RCPF-JDGSVD, RCPF-HJDGSVD, and
RIF-HJDGSVD. In Section~\ref{sec:4}, we develop thick-restart
schemes of these three JDGSVD algorithms with effective deflation and
purgation for computing several GSVD components of $(A,B)$.
Numerical experiments are presented in Section~\ref{sec:6} to
illustrate the performance of the three refined JDGSVD algorithms
and to make comparisons of them and the
three standard and harmonic extraction-based JDGSVD algorithms.
Finally, we conclude the paper in Section~\ref{sec:7}.

\section{The standard and two harmonic extraction-based JDGSVD methods}\label{sec:2}
We review CPF-JDGSVD, CPF-HJDGSVD and IF-HJDGSVD in
\cite{huang2022harmonic,huang2023cross} for
computing $(\alpha_*,\beta_*,u_*,v_*,x_*): =
(\alpha_1,\beta_1,u_1,v_1,x_1)$. Section~\ref{subsec:2-1} is devoted to the
construction and expansion of the searching subspaces,
Section~\ref{subsec:2-2} reviews the standard extraction,
Section~\ref{subsec:2-3} is on the CPF-harmonic extraction,
and Section~\ref{subsec:2-4} describes the IF-harmonic extraction.

\subsection{The construction and expansion of
searching subspaces}\label{subsec:2-1}

Assume that a $k$-dimensional right searching subspace
$\X\subset\mathbb{R}^{n}$ is available, from which an approximation
to $x_*$ is sought. Then we construct the two left searching
subspaces
\begin{equation}\label{search}
	\UU=A \X
	\qquad\mbox{and}\qquad
	\V=B\X,
\end{equation}
from which approximations to $u_*$ and $v_*$ are extracted, respectively.
Theorem~2.1 of \cite{huang2023cross} shows that the distance
between $u_*$ and $\UU$ is as small as that between $x_*$
and $\X$ provided that $\alpha_*$ is not very small;
analogously, the distance between $v_*$ and $\V$ is as small
as that between $x_*$ and $\X$ if $\beta_*$ is not very small.
Therefore, for the GSVD components
corresponding to not very large or small generalized
singular values, the left searching subspaces $\UU$ and $\V$ constructed by
\eqref{search}  are as good as $\X$.

Let $\widetilde X\in\mathbb{R}^{n\times k}$ be
an orthonormal basis matrix of $\X$, and compute the thin QR factorizations
\begin{equation}\label{qrAXBX}
	A \widetilde X= \widetilde UR_{A}
	\qquad\mbox{and}\qquad
	B \widetilde X= \widetilde VR_{B}
\end{equation}
to obtain the orthonormal basis matrices
$\widetilde U\in\mathbb{R}^{m\times k}$ and
$\widetilde V\in\mathbb{R}^{p\times k}$
of $\UU$ and $\V$. With $\UU$, $\V$ and $\X$ as well as
their orthonormal bases available, we can use one of the following
six extraction approaches to compute an approximation to
the desired GSVD component $(\alpha_*,\beta_*,u_*,v_*,x_*)$:
(\romannumeral1) the standard extraction, as done in
the CPF-JDGSVD method \cite{huang2023cross} and will be
reviewed in Section~\ref{subsec:2-2};
(\romannumeral2) the CPF-harmonic extraction, as exploited
by the CPF-HJDGSVD method \cite{huang2022harmonic} and will
be sketched in Section~\ref{subsec:2-3};
(\romannumeral3) the IF-harmonic extraction, as adopted in the
IF-HJDGSVD method \cite{huang2022harmonic} and will be reviewed
in Section~\ref{subsec:2-4}; (\romannumeral4) the refined CPF
extraction; (\romannumeral5) the refined CPF-harmonic extraction;
(\romannumeral6) the refined IF-harmonic extraction.

In  Section~\ref{sec:3}, we shall propose the three extraction
approaches in (\romannumeral4)--(\romannumeral6).
Together with their extensions and restart schemes
for computing more than
one GSVD components that will be presented in Section~\ref{sec:4},
we will set up our complete algorithms, which constitute our
major contribution in this paper.

We temporarily denote by $(\tilde\alpha,\tilde\beta,\tilde u,
\tilde v,\tilde x)$ an approximation to the desired GSVD
component $(\alpha_*,\beta_*,u_*,v_*,x_*)$ computed
by any one of the six extraction approaches listed above,
where the positive scalar pair $(\tilde\alpha,\tilde\beta)$
is required to satisfy $\tilde\alpha^2+\tilde\beta^2=1$, and the
2-norm unit length vectors $\tilde u\in\UU$,
$\tilde v\in\V$ and the $(A^TA+B^TB)$-norm unit length
$\tilde x\in\X$ are required to satisfy $A\tilde x=\tilde\alpha \tilde u$
and $B\tilde x=\tilde\beta \tilde v$.
Therefore, for the GSVD problem of $(A,B)$, in terms of
\eqref{gsvdvector}, the GSVD residual of
$(\tilde\alpha, \tilde\beta, \tilde u, \tilde v,\tilde x)$ is
\begin{equation}\label{residual}
	r=r(\tilde\alpha,\tilde\beta,\tilde u,\tilde v,\tilde x)=
\tilde\beta A^T\tilde u-\tilde\alpha B^T\tilde v.
\end{equation}
Clearly, $(\tilde\alpha,\tilde\beta,\tilde u,
\tilde v,\tilde x)$ is an exact GSVD component of
$(A,B)$ if and only if $r=\bm{0}$. Let
$\textit{tol}>0$ be a user-prescribed stopping tolerance.  If
\begin{equation}\label{converge}
	\|r\|\leq (\tilde\beta\|A\|_1+\tilde\alpha\|B\|_1)\cdot \textit{tol},
\end{equation}
we stop the iterations and accept
$(\tilde\alpha,\tilde\beta,\tilde u,\tilde v,\tilde x)$
as a converged approximation to the desired
$(\alpha_*,\beta_*,u_*,v_*,x_*)$.
Throughout the paper, we denote by $\|\cdot\|$ and $\|\cdot\|_1$
the $2$- and $1$-norms of a matrix or vector, respectively.

If $(\tilde\alpha,\tilde\beta,\tilde u,
\tilde v,\tilde x)$ does not yet converge, a JDGSVD type
method first expands the right searching subspace $\X$, then
updates the left searching subspaces $\UU$ and $\V$ in the way
\eqref{search}. Specifically, notice that
\begin{equation}\label{defy}
	\tilde y=(A^TA+B^TB)\tilde x=\tilde\alpha A^T\tilde u+\tilde\beta B^T\tilde v
\end{equation}
satisfies $\tilde{y}^T\tilde x=1$. Therefore, $I-\tilde{y}\tilde{x}^T$
and $I-\tilde{x}\tilde{y}^T$ are oblique projectors. We
approximately solve the correction equation
\begin{equation}\label{cortau}
	(I-\tilde{y}\tilde{x}^T)(A^TA-\rho^2B^TB)(I-\tilde{x}\tilde{y}^T)t=-r
	\qquad\mbox{for}\qquad
	t\perp \tilde{y}
\end{equation}
using some Krylov subspace iterative method such as the
MINRES method \cite{saad2003}
with the fixed $\rho=\tau$ when the approximate GSVD
component $(\tilde\alpha,\tilde\beta,\tilde u,
\tilde v,\tilde x)$ is not yet reasonably good,
and then switch to solving problem \eqref{cortau} with the
adaptively changing $\rho=\tilde\theta:=\tilde\alpha/\tilde\beta$ if
\begin{equation}\label{fixtol}
	\|r\|\leq(\tilde\beta\|A\|_1+\tilde\alpha\|B\|_1)\cdot \textit{fixtol}
\end{equation}
for a user-prescribed tolerance $\textit{fixtol}>0$, say, $10^{-4}$.
Criterion \eqref{fixtol}
means that $(\tilde\alpha,\tilde\beta,\tilde u,
\tilde v,\tilde x)$ is already a
fairly good approximation to $(\alpha_*,\beta_*,u_*,v_*,x_*)$.

Iteratively solving the correction equations of form \eqref{cortau}
in the JDGSVD type methods is called the inner iterations, and
the extraction of approximate GSVD components with
respect to $\UU$, $\V$ and $\X$ is called the outer iterations.
It has been shown in \cite{huang2023cross} that solving the
correction equations with \emph{low or modest} accuracy generally
suffices to make the outer iterations of
the resulting inexact JD type GSVD algorithms well
mimic those of their exact counterparts where all the correction
equations are solved accurately.
Therefore, for the correction equations of form
\eqref{cortau} in the JDGSVD methods proposed in
\cite{huang2022harmonic,huang2023cross}
and in the new JDGSVD methods to be proposed in this paper, we adopt the
inner stopping criterion in \cite{huang2023cross},
and stop the inner iterations when the inner relative residual
norm $\|r_{in}\|$ of an approximate solution satisfies
\begin{equation}\label{inncov}
	\|r_{in}\|\leq\min\{2c\tilde\varepsilon,0.01\},
\end{equation}
where $\tilde\varepsilon\in[10^{-4},10^{-3}]$ is a user-prescribed
parameter and $c$ is a constant depending on the value of $\rho$
and all the approximate generalized singular values of $(A,B)$
computed by the underlying JDGSVD method during the
current outer iteration.

An approximate solution of \eqref{cortau}, still
denoted by $t$ for brevity, is utilized to expand $\X$ so as to
obtain the new $\X_{\new}={\rm span}\{\widetilde
X,t\}$, and the corresponding orthonormal basis matrix is
updated by
\begin{equation}\label{expandX}
	\widetilde X_{\mathrm{new}}=[\widetilde X,\  x_{+}]
	\qquad\mbox{with}\qquad
	x_{+}=\frac{(I-\widetilde{X}\widetilde{X}^T)  t}{\|
(I-\widetilde{X}\widetilde{X}^T) t\|},
\end{equation}
where $x_{+}$ is called an expansion vector.
Making use of \eqref{search} and \eqref{qrAXBX} gives rise to
the expanded left searching subspaces
$$
\UU_{\new}=A\X_{\new}=\mathrm{span}\{\widetilde U,A x_{+}\}
\mbox{\ \ and\ \ } \V_{\new}=B\X_{\new}=\mathrm{span}\{\widetilde V,B x_{+}\}.
$$
We obtain their orthonormal basis matrices
$\widetilde U_{\new}$ and $\widetilde
V_{\new}$ by updating the thin QR factorizations
\begin{eqnarray}
	A\widetilde X_{\new}&=&\widetilde U_{\new} \cdot R_{A,\new}=[\widetilde U, \tilde u_{+}]
\cdot \begin{bmatrix}R_A & r_A\\&\gamma_A\end{bmatrix}, \label{updateura}\\
	B\widetilde X_{\new}&=&\widetilde V_{\new} \hspace{0.05em} \cdot R_{B,\new}
=[\widetilde V,\hspace{0.1em} \tilde v_{+}] \cdot \begin{bmatrix}R_B & r_B\\&\gamma_B\end{bmatrix}, \label{updateurb}
\end{eqnarray}
where
\begin{eqnarray}
	r_A=\widetilde U^TAx_{+},\qquad
	\gamma_A=\|Ax_{+}-\widetilde Ur_A\|,\qquad
	\tilde u_{+}=\frac{Ax_{+}-\widetilde Ur_A}{\gamma_A}, \label{updatera} \\
	r_B=\widetilde V^TBx_{+},\qquad
	\gamma_B=\|Bx_{+}-\widetilde Vr_B\|,\qquad
	\tilde v_{+}=\frac{Bx_{+}-\widetilde Vr_B}{\gamma_A}. \label{updaterb}
\end{eqnarray}

We then compute a new and hopefully better approximation
$(\tilde\alpha,\tilde\beta,\tilde u, \tilde v,\tilde x)$ with respect to $\X_{\rm new}$ and $\UU_{\rm new}$,
$\V_{\rm new}$, and repeat the above process until
convergence occurs.

\subsection{The standard extraction approach}\label{subsec:2-2}
Given $k$-dimensional right and left searching subspaces $\X$ and $\UU$,
$\V$ of form \eqref{search}, the standard extraction
approach finds nonnegative pairs
$(\tilde\alpha,\tilde\beta)$ with $\tilde\alpha^2+\tilde\beta^2=1$,
unit length $\tilde u\in\UU$ and $\tilde v\in\V$,
and $(A^TA+B^TB)$-norm unit length $\tilde x\in\X$
satisfying the conditions
\begin{equation}\label{sjdgsvd}
	\left\{\begin{aligned}
		&A\tilde x=\tilde\alpha\tilde u,\\
		&B\tilde x=\tilde\beta\tilde v,\\
		&\tilde\beta A^T\tilde u-\tilde\alpha B^T\tilde v \perp\X.
	\end{aligned}\right.
\end{equation}
Write $\tilde\theta=\tilde\alpha/\tilde\beta$.
It is straightforward to justify that
$$
(A^TA-\tilde\theta^2 B^TB)\tilde x\perp\X,
$$
which is exactly the standard Rayleigh--Ritz projection
of the generalized eigenvalue problem of the
matrix pair $(A^TA,B^TB)$ onto the subspace $\X$, and each of the
$k$ pairs $(\tilde\theta^2,\tilde x)$ is a Ritz approximation.
Therefore, we call
$(\tilde\alpha,\tilde\beta,\tilde u,\tilde v,\tilde x)$ a Ritz approximation to $(\alpha_*,\beta_*,u_*,v_*,x_*)$ with
$(\tilde\alpha,\tilde\beta)$ or
$\tilde\theta=\frac{\tilde\alpha}{\tilde\beta}$ the
Ritz value and $\tilde u$, $\tilde v$
and $\tilde x$ the left and right Ritz vectors of $(A,B)$
with respect to the left and right subspaces,
respectively.

It is known from \eqref{qrAXBX} that the projection matrices
$\widetilde U^TA\widetilde X=R_A$ and $\widetilde V^TA\widetilde X=R_B$.
Write $\tilde u=\widetilde U\tilde e$, $\tilde v=\widetilde V\tilde f$
and $\tilde x=\widetilde X\tilde d$.
Then (\ref{sjdgsvd}) reduces to
\begin{equation}\label{sab}
			R_A\tilde d=\tilde\alpha\tilde e,\qquad
			R_B\tilde d=\tilde\beta\tilde f ,\qquad
			\tilde\beta R_A^T\widetilde e=\tilde\alpha R_B^T\tilde f,
	\end{equation}
which is the vector form of GSVD of $(R_A,R_B)$.
Therefore, in the extraction phase, the standard extraction-based
CPF-JDGSVD method computes the GSVD of the $k$-by-$k$ matrix pair
$(R_A,R_B)$, picks up the GSVD component
$(\tilde\alpha,\tilde\beta,\tilde e,\tilde f,\tilde d)$
corresponding to the generalized singular value $\tilde\theta=\frac{\tilde\alpha}{\tilde\beta}$
closest to the target $\tau$, and takes
\begin{equation*}
	(\tilde\alpha,\tilde\beta,\tilde u,\tilde v,\tilde x)=
	(\tilde\alpha,\tilde\beta,\widetilde U\tilde e,\widetilde V\tilde f,\widetilde X\tilde d)
\end{equation*}
as an approximation to the desired GSVD component $(\alpha_*,\beta_*,u_*,v_*,x_*)$ of $(A,B)$.

\subsection{The CPF-harmonic extraction approach}\label{subsec:2-3}
For $B$ of full column rank, let
$B^TB=LL^T$ be the Cholesky factorization of $B^TB$.
It is proven in \cite{huang2022harmonic} that $(\sigma_*,u_*,z_*)$ with
$z_*=\frac{1}{\beta_*}L^Tx_{*}$ is a singular triplet of the matrix
\begin{equation}\label{deftildeA}
	\check{A}=AL^{-T}.
\end{equation}
Take the $k$-dimensional $\UU$ and $\Z=L^T\X$ as the left and
right searching subspaces for the left and right singular vectors
$u_*$ and $z_*$ of $\check A$, respectively. We note
that $\widetilde{Z}=L^TX$ is a basis matrix of $\Z=L^T\X$.
Then the CPF-harmonic extraction \cite{huang2022harmonic}
finds positive scalars $\phi>0$
and vectors $\check u\in\UU$ and $\check z\in\Z$ such that
\begin{equation}\label{cpfharmonic}
\begin{bmatrix}
	0 &\check A^T\\\check A& 0
\end{bmatrix}
\begin{bmatrix}
	\check z\\\check u
\end{bmatrix}
-\phi
\begin{bmatrix}
	\check z\\\check u
\end{bmatrix}
\ \perp\
\left(\begin{bmatrix}
	0 &\check A^T\\\check A& 0
\end{bmatrix}
-\tau I\right) \cdot
\mathcal{R}\left(\begin{bmatrix}
	\widetilde Z&\\& \widetilde  U
\end{bmatrix}\right),
\end{equation}
where $\mathcal{R}(\cdot)$ denotes the range space of a matrix.
This is the harmonic extraction approach
for the eigenvalue problem of the augmented matrix
$\bsmallmatrix{ &\check A^T\\\check A& }$
with respect to the searching subspace
$\mathcal{R}\left(\bsmallmatrix{\widetilde Z&\\&\widetilde U} \right)$
and the given target $\tau>0$;
see \cite{stewart2001matrix,vandervorst}.

Write $\check z=\widetilde Z\check d$ and $\check u=\widetilde U\check e$.
It is shown in \cite{huang2022harmonic} that \eqref{cpfharmonic} amounts to the
following symmetric generalized eigenvalue problem:
\begin{equation}\label{cpfeq20}
\begin{bmatrix}
	R_A^TR_A\!+\!\tau^2R_B^TR_B\!\!\!
	& -2\tau R_A^T
	\\-2\tau R_A
	&\!\!\! \widetilde U^T\!\!A(B^T\!B)^{-1}\!\!A^T\widetilde U\!+\!\tau^2I
\end{bmatrix}\!\!
\begin{bmatrix} \check d\\\check e\end{bmatrix}
=(\phi\!-\!\tau)\!\!
\begin{bmatrix}
	-\tau R_B^TR_B\!\!\! &R_A^T\\R_A &\!\!\!-\tau I
\end{bmatrix}\!\!
\begin{bmatrix}\check d\\\check e\end{bmatrix}.
\end{equation}
Denote by $H_{\mathrm{c}}$ and $G_{\mathrm{c}}$ the $2k\times 2k$
symmetric matrices in the
left and right hand sides of the above equation, respectively.
Computationally, suppose that $(B^TB)^{-1}=(LL^T)^{-1}$ can be
efficiently applied to obtain $H_{\mathrm{c}}$.
The CPF-harmonic extraction approach computes the
generalized eigendecomposition of the symmetric positive definite
matrix pair $(G_{\mathrm{c}},H_{\mathrm{c}})$, picks up the largest
generalized eigenvalue $\nu$ in magnitude and the corresponding
eigenvector $\bsmallmatrix{\check d\\\check e}$, and
takes
\begin{equation}\label{cpfh}
(\phi,\check u,\check z)=\left(\tau+\frac{1}{\mu},\frac{\widetilde U\check
	e}{\|\check e\|},\frac{\widetilde Z\check d}{\|\widetilde Z\check d\|}\right)
\end{equation}
as an approximation to the singular triplet $(\sigma_*,u_*,z_*)$ of $\check A$.

Note that the exact right generalized singular vector
$x_*=\beta_*L^{-T}z_*$.
Therefore, we take $L^{-T}\check z=L^{-T}\widetilde Z\check d
=\widetilde X\check d$ as the approximation to $x_*$ in direction.
Concretely, we take the approximate right generalized singular vector
$\check x$ of $(A,B)$ as
\begin{equation}\label{cpfx}
\check x=\frac{1}{\check\delta}\widetilde X\check d
\qquad\mbox{with}\qquad
\check\delta=\sqrt{\|\check e\|^2+\|\check f\|^2},
\end{equation}
where $\check e$ is recomputed by $\check e=R_A\check d$ and $\check
f=R_B\check d$. With such $\check\delta$, the approximate $\check x$ is of
$(A^TA+B^TB)$-norm unit length \cite{huang2022harmonic}.
We then take the new approximate
generalized singular value and left generalized singular vectors as \begin{equation}\label{checksigmauv}
\check\alpha=\frac{\|\check e\|}{\check\delta},\qquad
\check\beta=\frac{\|\check f\|}{\check\delta}
\qquad\mbox{and}\qquad
\check u=\frac{\widetilde U\check e}{\|\check e\|},\qquad
\check v=\frac{\widetilde V\check f}{\|\check f\|},
\end{equation}
which are called the CPF-harmonic Ritz approximations and satisfy $
A\check x=\check \alpha \check u$  and $
B\check x=\check\beta\check v
$
with
$\check\alpha^2+\check\beta^2=\|\check u\|=\|\check v\|=1$.
Moreover, it is known from \cite{huang2022harmonic} that the new  $\check\theta=\frac{\check\alpha}{\check\beta}$
is a better approximation to $\sigma_*$ than
$\phi$ in \eqref{cpfh} in the sense that
\begin{equation}\label{rayqu}
\|(A^TA-\check\theta^2B^TB)\check x\|_{(B^TB)^{-1}}
\leq\|(A^TA-\phi^2B^TB)\check x\|_{(B^TB)^{-1}},
\end{equation}
where $\|\cdot\|_M$ is the $M$-norm for a symmetric positive definite matrix $M$.

\subsection{The IF-harmonic extraction approach}\label{subsec:2-4}
For a general and possibly rank deficient $B$,
the CPF-harmonic extraction
does not work. Alternatively, the IF-harmonic
extraction \cite{huang2022harmonic} is for more general purpose. It
finds approximate generalized singular
values $\varphi>0$ and approximate right generalized singular
vectors $\hat x\in\X$ with $\|\hat x\|_{A^TA+B^TB}=1$ such that
\begin{equation}\label{ifharmonic}
(A^TA-\varphi^2B^TB)\hat x\ \perp\ (A^TA-\tau^2B^TB)\X.
\end{equation}
This is precisely the harmonic Rayleigh--Ritz projection on the generalized eigenvalue problem of
$(A^TA,B^TB)$ with respect to the subspace $\X$ and the given
target $\tau^2$.

Write $\hat x=\frac{1}{\hat\delta}\widetilde X\hat d$
with $\|\hat d\|=1$ and $\hat\delta$
a normalizing parameter to be determined.
Then requirement \eqref{ifharmonic} is equivalent to
\begin{equation}\label{harmoniceq}
\widetilde {X}^T(A^TA-\tau^2B^TB)^2\widetilde X\hat d
=(\varphi^2-\tau^2)\widetilde{X}^T(A^TA-\tau^2B^TB)
B^TB\widetilde X\hat d.
\end{equation}
Denote by $H_{\tau}$ and
$G_{\tau}$ the $k$-by-$k$ matrices in the left and right hand sides
of the above equation, respectively. Then $(\varphi^2-\tau^2)$ is a generalized
eigenvalue of the matrix pair $(H_{\tau},G_{\tau})$ with $\hat d$
the corresponding unit length generalized eigenvector.
The IF-harmonic extraction approach computes  the generalized
eigendecomposition of $(G_{\tau},H_{\tau})$,
and picks up the generalized eigenpair $(\nu,\hat d)$ corresponding
to $\varphi=\sqrt{\tau^2+\frac{1}{\nu}}$ closest to $\tau$ among all the
eigenpairs of $(G_{\tau},H_{\tau})$.
Then $\varphi$ is the IF-harmonic Ritz value that approximates
the desired $\sigma_*$ and
\begin{equation}\label{hatx}
\hat x=\frac{1}{\hat\delta}\widetilde X\hat d
\qquad\mbox{with}\qquad
\hat\delta=\sqrt{\|\hat e\|^2+\|\hat f\|^2}
\end{equation}
is the corresponding right IF-harmonic Ritz vector that approximates
the desired $x_*$, where $\hat e=R_A\hat d$ and $\hat f=R_B\hat d$. Such
$\hat\delta$ guarantees that $\hat x$ is of $(A^TA+B^TB)$-norm unit length.
The corresponding IF-harmonic Ritz values and left IF-harmonic Ritz vectors are computed by
\begin{equation}\label{hatsigmauv}
\hat\alpha=\frac{\|\hat e\|}{\hat \delta},\qquad
\hat\beta=\frac{\|\hat f\|}{\hat \delta}
\qquad\mbox{and}\qquad
\hat u=\frac{\widetilde U\hat e}{\|\hat e\|},\qquad
\hat v=\frac{\widetilde V\hat f}{\|\hat f\|}.\qquad
\end{equation}

It is known from \cite{huang2022harmonic} that the IF-harmonic Ritz
approximation $(\hat\alpha,\hat\beta,\hat u,\hat v,\hat x)$
satisfies $A\hat x=\hat\alpha\hat u$ and $
B\hat x=\hat\beta\hat v$
with
$\hat \alpha^2+\hat\beta^2=\|\hat u\|=\|\hat v\|=1$.
The IF-harmonic approximate generalized singular value
$\hat\theta=\frac{\hat\alpha}{\hat\beta}$ is better than the above $\varphi$
in the sense of \eqref{rayqu} where $\phi$, $\check\theta$ and $\check x$ are
replaced by $\varphi$, $\hat\theta$ and $\hat x$, respectively.

\section{The refined JDGSVD type methods}\label{sec:3}
For an approximation to the desired GSVD component
$(\alpha_*,\beta_*,u_*,v_*,x_*)$ obtained by the standard or
harmonic extraction approaches described in Section~\ref{sec:2},
in this section we will unify the notation and denote the approximation
by $(\alpha,\beta,u,v,x)$ with $\theta=\frac{\alpha}{\beta}$.
We now propose a refined extraction approach:
Find a unit length vector $ x_r\in\X$ satisfying the optimal requirement
\begin{equation}\label{refined}
	\|(A^TA-\theta^2B^TB)x_r\| =
	\min_{w\in\X,\|w\|=1}\|(A^TA-\theta^2B^TB)w\|,
\end{equation}
rescale $x_r$ to $\bar x$ with $\|\bar x\|_{A^TA+B^TB}=1$, and use
$\bar x$ to approximate $x_*$. We call $\bar x$ a refined
approximate right generalized singular vector, or simply the refined or
refined harmonic right Ritz vector, of $(A,B)$ over the subspace $\X$
if $\theta$ is a Ritz value or harmonic Ritz value obtained by the standard
or harmonic extraction approaches.

Jia \cite{jia2004some} has proven that if
$(\theta^2,\bar{x})$ is not an exact eigenpair of $(A^TA,B^TB)$
then
$\|(A^TA-\theta^2B^TB)\bar x\|<\|(A^TA-\theta^2B^TB)x\|$ strictly;
moreover, if there is another standard or harmonic
Ritz value close to $\theta$, then
$$
\|(A^TA-\theta^2B^TB)\bar x\|\ll\|(A^TA-\theta^2B^TB)x\|
$$
generally holds,
meaning that $\bar x$ may be much more accurate
than $x$ as an approximation to $x_{*}$. Most importantly,
the fundamental convergence results in \cite{jia2005,jiastewart2001}
indicate that there is a Ritz value $\theta\rightarrow \sigma_*$
unconditionally and
the convergence of the refined Ritz or harmonic Ritz vector
$\bar{x}\rightarrow x_*$ is guaranteed too
once the distance between $x_*$ and $\X$ tends to zero, while
the Ritz or (CPF or IF-)harmonic Ritz vector $x$ may
converge erratically and even may fail to converge even if
$\X$ contains sufficiently accurate approximations to $x_*$.

We now consider the accurate and efficient computation of $\bar{x}$.
For an arbitrary unit length vector $w\in\X$, write
$w=\widetilde Xd$ with $\|d\|=1$. Then $x_r=\widetilde X\bar d$
with $\|\bar d\|=1$, and the minimization problem
\eqref{refined} is equivalent to
\begin{equation}\label{mineqmin}
	\|(A^TA-\theta^2B^TB)\widetilde X\bar d\|=
	\min_{d\in\mathbb{R}^{k}, \|d\|=1}\|(A^TA-\theta^2B^TB)\widetilde Xd\|.
\end{equation}
Therefore, $\bar d$ is the right singular vector of
$G_{\theta}=(A^TA-\theta^2B^TB)\widetilde X$ corresponding to
its smallest singular value, and it is also the eigenvector
of the cross-product matrix
\begin{equation}\label{defGr}
	H_{\theta}=G_{\theta}^TG_{\theta}
	=\widetilde{X}^T(A^TA-\theta^2B^TB)^2\widetilde X
\end{equation}
associated with its smallest eigenvalue.
Jia \cite{jia2006using} has proposed and developed a variant of
the cross product-based QR algorithm for the accurate SVD
computation of a general matrix. In our specific context, we only need to
use the standard QR algorithm to obtain $\bar{d}$.
Remarkably, Jia \cite{jia2006using} has shown that
the cross product-based QR algorithm
is much more efficient than the standard
Golub--Kahan and Chan SVD
algorithms \cite{golub2012matrix,stewart2001matrix} applied
to $G_{\theta}$ for $n\gg k$.

More precisely, in finite precision
arithmetic, Jia \cite{jia2006using} has proven
that, provided that the smallest singular value of
$G_{\theta}$ is well separated from its second smallest one
then the computed $\bar d$'s by the cross product-based QR algorithm
and the Golub--Kahan or Chan SVD algorithm
essentially have the same accuracy; when the computed smallest
singular value of $G_{\theta}$ is taken as the square root
of the Rayleigh quotient $x^TH_{\theta}x$ that is calculated
by the formula $(G_{\theta}\bar d)^T(G_{\theta}\bar d)$
with $\bar d$ the {\em computed} eigenvector of $H_{\theta}$,
it has the same accuracy as the smallest singular value computed by the
standard Golub--Kahan or Chan SVD algorithm applied to $G_{\theta}$;
see \cite{jia2006using} for a detailed analysis and comparison.
We should particularly remind that in our context the smallest
singular value of $G_{\theta}$ tends to zero as
$\theta\rightarrow \sigma_*$ but the second smallest one of
$G_{\theta}$ is typically not small and thus is well separated
from the smallest one.

In view of the above, rather than computing the SVD of $G_{\theta}$
at expensive cost,
we compute the eigendecomposition of $H_{\theta}$ cheaply, and pick up
the desired eigenvector $\bar d$.
We remark that, as the subspaces are expanded, we can efficiently form
$H_{\theta}$ by
\begin{equation}\label{comGr}
	H_{\theta}=H_A+\theta^4H_B-\theta^2(H_{A,B}^T+H_{A,B}),
\end{equation}
where the intermediate matrices
\begin{equation}\label{defHAB}
H_A= \widetilde{X}^T(A^TA)^2\widetilde X,\qquad
H_B= \widetilde{X}^T(B^TB)^2\widetilde X,\qquad
H_{A,B}=\widetilde{X}^TA^TAB^TB\widetilde X
\end{equation}
can be efficiently updated at each step and they are also used to
efficiently form the projected matrices $G_{\tau}=H_{A,B}-\tau^2H_B$ and
$H_{\tau}=H_A+\tau^4H_B-\tau^2(H_{A,B}^T+H_{A,B})$
involved in the IF-harmonic extraction approach;
see \eqref{harmoniceq}. Also, it is important to notice
that when $\ell$ GSVD components are required, which will
be considered in the next section, we can efficiently
form all the cross-product
matrices $H_{\theta}$ in \eqref{comGr} for different $\theta$'s.

By definition, we need to rescale $x_r$ in
\eqref{refined} to obtain the refined or refined harmonic
Ritz vector $\bar x$ with $\|\bar  x\|_{A^TA+B^TB}=1$.
Write $$\bar x=\frac{1}{\bar\delta}\widetilde X\bar d,$$ where $\bar\delta$
is a normalizing factor to be determined.
Following the same derivations as in Section~\ref{sec:2},
we have
$$\bar\delta=\sqrt{\|\bar e\|^2+\|\bar f\|}
\qquad
\mbox{with}\qquad
\bar e=R_A\bar d
\quad\mbox{and}\quad\bar f=R_B\bar d,$$
where $R_A$ and $R_B$ are defined in \eqref{qrAXBX}.
Analogously to those done in the CPF- and IF-harmonic
extraction approaches, we compute the refined or refined harmonic Ritz values
and the two refined or refined harmonic left Ritz vectors by
\begin{equation}\label{refinegsvd}
	\bar\alpha=\frac{\|\bar e\|}{\bar\delta},\qquad
	\bar\beta=\frac{\|\bar f\|}{\bar\delta}
\qquad \mbox{and}\qquad
	\bar u=\frac{\widetilde U\bar e}{\|\bar e\|},\qquad
	\bar v=\frac{\widetilde V\bar f}{\|\bar f\|}.
\end{equation}
It is easily verified that
$$
A\bar x=\bar\alpha\bar u \qquad\mbox{and}\qquad
B\bar x=\bar\beta\bar v
$$
with $\bar\alpha^2+\bar\beta^2=\|\bar u\|=\|\bar v\|=1$.
The quintuple $(\bar\alpha,\bar\beta,\bar u,\bar v,\bar x)$
is called a refined or refined CPF-harmonic or refined
IF-harmonic approximation to the desired GSVD component
$(\alpha_*,\beta_*,u_*,v_*,x_*)$ of $(A,B)$, depending on by which
extraction approach $\theta$ in \eqref{refined} is computed:
the standard extraction, the CPF-harmonic or IF-harmonic
extraction. Particularly, $(\bar\alpha,\bar\beta)$ or
$\bar\theta=\frac{\bar\alpha}{\bar\beta}$ is called the
refined, refined CPF-harmonic, or refined IF-harmonic
Ritz value and, correspondingly, $\bar u$, $\bar v$ and
$\bar x$ are called the refined, refined CPF-harmonic or
IF-harmonic left and right Ritz vectors of $(A,B)$.

Combining the refined, refined CPF-harmonic and refined
IF-harmonic extraction approaches with the subspace
expansion approach described in Section~\ref{subsec:2-1},
we have now proposed the refined CPF, refined CPF-harmonic and refined
IF-harmonic JDGSVD methods, i.e., RCPF-JDGSVD,
RCPF-HJDGSVD and RIF-HJDGSVD, respectively.
For each of these methods, the corresponding residual and
stopping criterion are the same as \eqref{residual} and \eqref{converge}.

\section{Thick-restart refined JDGSVD algorithms with deflation and purgation} \label{sec:4}
In this section,  by introducing appropriate deflation and
purgation techniques, we develop practical thick-restart refined JDGSVD
algorithms for solving Problem~\ref{probl}.

\subsection{Thick-restart}\label{subsec:5-1}
As the dimension $k$ of searching subspaces increases,
the storage requirements and computational costs
of the previous basic JDGSVD type algorithms become unaffordable.
When $k$ reaches the maximum number $k_{\max}$ allowed but
the algorithms do not yet converge, it is necessary to
restart them. To this end, we adopt the thick-restart technique, which was
first advocated in \cite{stath1998} for the eigenvalue problem
and has been nontrivially extended to the SVD and GSVD problems
in \cite{huang2019inner,huang2022harmonic,huang2023cross,
stath1998,wu16primme,wu2015preconditioned}.
Specifically, the thick-restart
takes certain $k_{\min}$-dimensional subspaces
of the current left and right searching subspaces as the initial
left and right ones for the next cycle, so that they retain as
much information as possible on the desired and $k_{\min}-1$
nearby GSVD components of $(A,B)$.  We then expand them step
by step in the way described in Section~\ref{sec:2},
compute new approximate GSVD components with respect to
the expanded subspaces at each
step, and check the convergence.
If converged, we stop; otherwise, we repeat the same
process until the dimension of expanded subspaces reaches $k_{\max}$.

We will present an efficient and stable computational procedure
for the thick-restart.
For each of RCPF-JDGSVD, RCPF-HJDGSVD and RIF-HJDGSVD,
when $k=k_{\max}$ but the method does not yet converge,
we compute $k_{\min}$ approximate
right generalized singular vectors $\bar x_i=\widetilde X\bar d_i$,
$i=1,\dots,k_{\min}$ associated with the
$k_{\min}$ approximate singular values closest to
$\tau$, where $\bar x_1$ is selected as the
approximation to the desired $x_*$.
Then the thick-restart takes the new initial right searching subspace
$\X_{\new}=\spans\{\bar x_1,\dots,\bar x_{\min}\}$.

Denote $D_1=[\bar d_1,\dots,\bar d_{k_{\min}}]$.
Next we show how to efficiently obtain the $k_{\min}$-dimensional
new right and left subspaces $\mathcal{X}_{\rm new}$,
$\mathcal{U}_{\rm new}$ and $\mathcal{V}_{\rm new}$ as well as
their orthonormal bases, whose computational details
were not described in \cite{huang2022harmonic,huang2023cross}.

Compute the thin QR factorization $D_1=Q_dR_d$
using $\mathcal{O}(k_{\max}k_{\min}^2)$ flops.
Then
\begin{equation}\label{xnew}
	 \widetilde{X}_{\new}=\widetilde XQ_d
\end{equation}
is the orthonormal basis matrix of $\X_{\new}$, whose computation
costs $2nk_{\max}k_{\min}$ flops. As for the new
$\UU_{\new}=A\X_{\new}$ and $\V_{\new}=B\X_{\new}$, we compute the
thin QR factorizations of the small sized matrices
\begin{equation}\label{qrsmall}
	R_AQ_d=Q_eR_{A,\new}
	\qquad \mbox{and} \qquad
	R_BQ_d=Q_fR_{B,\new}
\end{equation}
using $\mathcal{O}(k_{\max}^2k_{\min})$ flops,
where $R_A$, $R_B$ are defined as in \eqref{qrAXBX}.
Exploiting \eqref{qrAXBX} and \eqref{xnew}--\eqref{qrsmall},
we obtain the thin QR factorizations
\begin{eqnarray*}
	 &&\hspace{0.1em} A\widetilde{X}_{\new} = \hspace{0.1em}A
\widetilde{X}Q_d=\hspace{0.1em}\widetilde UR_AQ_d=(\widetilde UQ_e)R_{A,\new},\\
	 &&B\widetilde{X}_{\new} = B\widetilde{X}Q_d=\widetilde VR_BQ_d=(\widetilde VQ_f)R_{B,\new},
\end{eqnarray*}
showing that the columns of
\begin{equation*}
	\widetilde U_{\new}=\widetilde UQ_e
	\qquad\mbox{and}\qquad
	\widetilde V_{\new}=\widetilde VQ_f
\end{equation*}
form orthonormal bases of $\UU_{\new}$ and $\V_{\new}$, respectively,
whose computation costs $2(m+p)k_{\max}k_{\min}$ flops.
Together with the computation of  $\widetilde{X}_{\new}$,
the above whole process approximately costs $2(m+n+p)k_{\max}k_{\min}$ flops
since $m,p,n\gg k_{\max}>k_{\min}$.

For the intermediate matrices in \eqref{defHAB} used to form
$H_{\theta}$ defined by \eqref{comGr}, as
has been done in the IF-harmonic JDGSVD algorithm, we efficiently
update them by
\begin{equation}\label{inter}
	H_{A,\new}=Q_d^TH_AQ_d,\qquad
	H_{B,\new}=Q_d^TH_BQ_d,\qquad
	H_{A,B,\new}=Q_d^TH_{A,B}Q_d
\end{equation}
at cost of $\mathcal{O}(k_{\max}^2k_{\min})$ flops.

Summarizing the above,
the construction cost of orthonormal bases of the $k_{\min}$-dimensional
$\UU_{\new}$, $\V_{\new}$, $\X_{\new}$ and the
matrices in \eqref{inter} is approximately $2(m+p+n)k_{\max}k_{\min}$
flops. Typically, in computation,
one takes $k_{\max}=20\sim 30$ and $k_{\min}=3\sim 5$; see
\cite{huang2022harmonic,huang2023cross}. Therefore, forming
the restarting initial left and right searching
subspaces in the thick-restart is very cheap.

\subsection{Deflation and purgation}\label{subsec:5-2}
We can adapt the effective and
efficient deflation techniques proposed in
\cite{huang2022harmonic,huang2023cross} to the thick-restart
RCPF-JDGSVD, RCPF-HJDGSVD and RIF-HJDGSVD. We will
elaborate some key elements of subspaces during deflation
that were not given adequate and very clear arguments in
\cite{huang2022harmonic,huang2023cross},
which turn out to play a crucial role in both mathematics
and effective and efficient implementations of JDGSVD type
algorithms for computing more than one GSVD components.

Suppose that one of the RCPF-JDGSVD, RCPF-HJDGSVD and RIF-HJDGSVD algorithms
has computed $j(<\ell)$ converged approximations
$(\alpha_{i,c},\beta_{i,c},u_{i,c},v_{i,c},x_{i,c})$ to
$(\alpha_i,\beta_i,u_i,v_i,x_i)$, whose residual norms satisfy
\begin{equation}\label{stopcrit}
	\|r_i\|=\|\beta_{i,c} A^Tu_{i,c}\!-\!\alpha_{i,c}B^Tv_{i,c}\|
	\leq(\beta_{i,c}\|A\|_1\!+\!\alpha_{i,c}\|B\|_1)\cdot tol, \quad i=1,\dots,j.
\end{equation}
Denote
\begin{eqnarray}
&&U_c\hspace{0.1em}=[u_{1,c}\hspace{0.01em},\dots, u_{j,c}\hspace{0.01em}], \qquad\quad
C_c=\diag\{\alpha_{1,c},\dots,\alpha_{j,c}\},\nonumber\\
&&V_c\hspace{0.2em}=[v_{1,c}\hspace{0.11em},\dots, v_{j,c}\hspace{0.11em}],\qquad\quad
S_c\hspace{0.07em}=\diag\{\beta_{1,c}\hspace{0.11em},\dots,\beta_{j,c}
\hspace{0.11em}\},\nonumber\\
&&X_c=[x_{1,c},\dots, x_{j,c}],\qquad\quad
Y_c\hspace{0.09em}=(A^TA+B^TB)X_c. \label{convx}
\end{eqnarray}
Then
$$
AX_c=U_cC_c,\quad BX_c=V_cS_c,\quad C_c^2+S_c^2=I_j,\quad
Y_c=A^TU_cC_c+B^TV_cS_c,
$$
and the  F-norm of the residual matrix satisfies
\begin{equation*}
	\|R_c\|_F=\|A^TU_cS_c-B^TV_cC_c\|_F\leq\sqrt{j(\|A\|_1^2+\|B\|_1^2)}\cdot \textit{tol}.
\end{equation*}

Suppose that the current $X_c$ and $Y_c$ are bi-orthogonal, i.e.,
$Y_c^TX_c=I_j$. We point out that this bi-orthogonality is
fulfilled in our six JDGSVD type algorithms.
It is known from Proposition~4.1 of \cite{huang2023cross} that
$(\alpha_i,\beta_i,u_i,v_i,x_i),\ i=j+1,\ldots,q$ are the exact GSVD
components of the deflated matrix pair
\begin{equation}\label{defpair}
	(A(I-X_cY_c^T),B(I-X_cY_c^T))
\end{equation}
restricted to the range space of the oblique projector $(I-X_cY_c^T)$
if $\textit{tol}=0$ in (\ref{stopcrit}).
Therefore, we can apply any one of RCPF-JDGSVD, RCPF-HJDGSVD and RIF-HJDGSVD
to the deflated matrix pair in \eqref{defpair} to compute
the next desired GSVD component
$(\alpha_{*},\beta_{*},u_{*},v_{*},x_{*}):=
(\alpha_{j+1},\beta_{j+1},u_{j+1},v_{j+1},x_{j+1})$ of $(A,B)$.

Remarkably, when
the converged $(\alpha_{j,c},\beta_{j,c},u_{j,c},v_{j,c},x_{j,c})$
has been found, the current
subspaces usually contain reasonably rich
information on $u_{*},v_{*},x_{*}$.
To make full use of such available information,
we present an effective and efficient purgation
technique in the thick-restart
RCPF-JDGSVD, RCPF-HJDGSVD and RIF-HJDGSVD algorithms.
Instead of constructing initial searching
subspaces from scratch, we purge the newly converged
$$
x_{j,c}:=x=\widetilde Xd
$$
from the current $\X$, and take the reduced subspace,
denoted by $\X_{\rm new}$, as an initial right searching subspace when
extracting an approximation to $(\alpha_{*},\beta_{*},u_{*},v_{*},x_{*})$.

Concretely, denote
$$
\widetilde{X}_{\new}=\widetilde XQ_D
$$
with some orthonormal matrix $Q_D\in\mathbb{R}^{k\times(k-1)}$ to be
determined.
We require that $\widetilde X_{\new}$ be
orthogonal to $Y_{c,\new}=[Y_{c},y]$, where $y=(A^TA+B^TB)x$.
Suppose that the current $\widetilde X$ is orthogonal to $Y_{c}$.
Then we only need to make $\widetilde X_{\new}$
orthogonal to $y$. By \eqref{qrAXBX}, this amounts to
$$
\widetilde X_{\new}^Ty=Q_D^T\widetilde{X}^T(A^TA+B^TB)\widetilde Xd
=Q_D^T(R_A^TR_A+R_B^TR_B)d=\bm{0}.
$$
Therefore, the columns of $Q_D$ form an orthonormal basis
of the orthogonal complement of $\mathcal{R}(d^{\prime})$ with
respect to $\mathbb{R}^k$, where $d^{\prime}=(R_A^TR_A+R_B^TR_B)d$.
In computation, we compute the full QR factorization of the
$k$-by-$1$ matrix $d^{\prime}$ using approximately $4k^2$ flops
\cite[p.249-250]{golub2012matrix},
take the second to last columns of its $Q$-factor to form $Q_D$,
and obtain the desired $\X_{\rm new}$.
Then following the analogous process to those
in Section~\ref{subsec:5-1}, by replacing $Q_d$ with $Q_D$,
we can carry out this purgation strategy to
obtain the reduced $\UU_{\rm new}$ and $\V_{\rm new}$
with little extra cost. We then extract an approximation
to $(\alpha_{*},\beta_{*},u_{*},v_{*},x_{*})$ with respect
to them, and expand the searching subspaces in the thick-restart
way described previously.

We now unify to denote by $(\alpha,\beta,u,v,x)$ the approximate
GSVD component obtained by one of RCPF-JDGSVD, RCPF-HJDGSVD
and RIF-HJDGSVD with respect to the left and right searching subspaces
$\UU$, $\V$ and $\X$. If the current $\X$ is
orthogonal to $\mathcal{R}(Y_c)$, then
$$
\UU=(A-X_cY_c^T)\X=A\X \mbox{\ \ and\ \ }
\V=(B-X_cY_c^T)\X=B\X.
$$
Therefore, for such an $\X$, in the JDGSVD type algorithms of
\cite{huang2022harmonic,huang2023cross} and this paper,
computationally, we never work on the matrix pair in \eqref{defpair}
explicitly, instead we always work on $(A,B)$ directly.
The gains are twofold: the resulting JDGSVD type algorithms are more efficient
at each step; in finite precision arithmetic,
they enable us to compute the approximations more
accurately. The first point is straightforward.
But the second point is subtle and quite complicated, and its
arguments and details are out of the
scope of this paper.

Recall that the reduced $\X_{\rm new}$ can be made orthogonal to
$\mathcal{R}(Y_c)$
and is an instance of the current subspace $\X$ when
computing $(\alpha_*,\beta_*,u_*,v_*,x_*)$. We present
the following important result.

\begin{theoremmy}
Suppose that the current right subspace $\X$ is
orthogonal to $\mathcal{R}(Y_c)$.
Then the expanded $\mathcal{X}$'s are also orthogonal to
$\mathcal{R}(Y_c)$ at subsequent expansion steps.
\end{theoremmy}

\begin{proof}
We only need to prove the assertion for one expansion step.
For the current $\X$, if $(\alpha,\beta,u,v,x)$ does not converge yet,
then in the expansion phase, the original correction equation
\eqref{cortau} for $\ell>1$ becomes (cf.
\cite{huang2022harmonic,huang2023cross})
\begin{equation}\label{deflat}
	(I-Y_pX_p^T)(A^TA-\rho^2B^TB)(I-X_pY_p^T)t=-(I-Y_cX_c^T)r
	\quad\mbox{for}\quad
	t\perp Y_p,
\end{equation}
where $r$ is the residual of $(\alpha,\beta,u,v,x)$ defined by
\eqref{residual}, $\rho=\tau$ or $\theta=\frac{\alpha}{\beta}$
(see \eqref{cortau} and the paragraph that follows), and
$X_p=[X_c,x]$, $Y_p=[Y_c,y]$ with $X_c$ and $Y_c$ defined in \eqref{convx}
and $y$ defined by \eqref{defy}.
It follows from $x\in\X\perp\mathcal{R}(Y_c)$ that
$X_p$ and $Y_p$ are bi-orthogonal:
$Y_p^TX_p=I_{j+1}$, meaning that $I-Y_pX_p^T$
and $I-X_pY_p^T$ are oblique projectors.
With an approximate solution $t$ of \eqref{deflat} found,
we orthonormalize it against the orthonormal basis matrix
$\widetilde X$ of $\X$ to obtain the expansion vector
$x_{+}$ and update the basis matrices
$\widetilde X$, $\widetilde U$ and $\widetilde V$ by
\eqref{expandX}--\eqref{updaterb}.
By the way that $\widetilde X$ is augmented,
the resulting expanded $\X$ is automatically orthogonal
to $\mathcal{R}(Y_c)$ at the expansion step.
\end{proof}

The assertion in this theorem is
crucial and ensures that we always work on $(A,B)$ rather than
the explicitly deflated matrix pair in \eqref{defpair}
when using each of the six JDGSVD type algorithms to
computing more than one GSVD components.

Once $(\alpha,\beta,u,v,x)$ has converged, we add it to the
previous converged partial GSVD $(C_c,S_c,U_c,V_c,X_c)$
of $(A,B)$, and set $j:=j+1$.
The RCPF-JDGSVD, RCPF-HJDGSVD and RIF-HJDGSVD algorithms
proceed in this way until all the $\ell$ desired GSVD
components of $(A,B)$ are found, and
the computed $X_c=[x_{1,c},\ldots,x_{\ell,c}]$
and $Y_c=[y_{1,c},\ldots,y_{\ell,c}]$ satisfy $Y_c^TX_c=I_{\ell}$.

\begin{algorithm}[tbph]
	\caption{The thick-restart RCPF-JDGSVD, RCPF-HJDGSVD,
		RIF-HJDGSVD algorithms with deflation and purgation for
		the target $\tau$.}
	\renewcommand{\algorithmicrequire}{\textbf{Input:}}
	\renewcommand{\algorithmicensure} {\textbf{Output:}}
	\begin{algorithmic}[1]\label{algorithm:1}
		\STATE{\textbf{Initialization:}\ Set $k=1$, $k_c=0$, and initialize $C_c=[\ ], S_c=[\ ], U_c=[\ ],
			V_c=[\ ],X_c=[\ ]$, $Y_c=[\ ]$.
			Set $\widetilde U=[\ ]$, $\widetilde V=[\ ]$,
			$\widetilde X=[\ ]$, and choose a unit-length starting vector $ x_{+}=x_0$.}
		
		\WHILE{$k\geq0$}
		
		\STATE{\label{step:3}Set $\widetilde X=[\widetilde X,  x_{+}]$,
			and update the QR factorizations
			$A\widetilde X=\widetilde UR_A$,
			$B\widetilde X=\widetilde VR_B$.}

		\STATE{\label{step:4}
			\textbf{RCPF-JDGSVD:} Compute the GSVD of $(R_A,R_B)$, and pick
			up the generalized singular value $\theta$ closest to the target $\tau$.
			
			\textbf{RCPF-HJDGSVD:} Form $H_c$ and $G_c$ defined in \eqref{cpfeq20}.
Compute the generalized eigendecomposition of $(G_c,H_c)$,
pick up the eigenvector $d$ corresponding to the largest eigenvalue $\mu$ in magnitude,
and compute the approximate $\theta=\frac{\|R_Ad\|}{\|R_Bd\|}$.
			
			\textbf{RIF-HJDGSVD:} Form $G_{\tau}$ and $H_{\tau}$ defined in \eqref{harmoniceq}.
Compute the generalized eigendecomposition of $(G_{\tau},H_{\tau})$,
pick up the eigenvector $d$ associated with the eigenvalue $\nu$
such that $\sqrt{\tau^2+\frac{1}{\nu}} $ is closest to $\tau$, and
compute $\theta=\frac{\|R_Ad\|}{\|R_Bd\|}$.}
		
		\STATE{Form $H_{\theta}$ defined by \eqref{comGr}, compute the
eigendecomposition of $H_{\theta}$,
and pick up the eigenvector $\bar d$ corresponding to the smallest eigenvalue.}
		
		\STATE{\label{step:5}Form $\bar e=R_A\bar d$, $\bar f=R_B\bar d$,
and calculate $\bar \delta=\sqrt{\|\bar e\|^2+\|\bar f\|^2}$.
			Compute the approximate GSVD component
			$(\bar \alpha,\bar \beta,\bar  u,\bar v,\bar  x) =\left(\dfrac{\|\bar e\|}{\bar \delta},
			\dfrac{\|\bar f\|}{\bar \delta},
			\dfrac{1}{\bar \delta}\widetilde X\bar d,
			\dfrac{\widetilde U\bar e}{\|\bar e\|},
			\dfrac{\widetilde V\bar f}{\|\bar f\|} \right)$,
			the vector $\bar y=\bar \alpha A^T\bar u+\bar \beta B^T\bar  v$,
and the residual $r=\bar \beta A^T\bar u-\bar \alpha B^T\bar v$.}
		
		\IF{$\|r\|\leq (\bar\beta \|A\|_1+\bar\alpha \|B\|_1)\cdot tol$}
		
		\STATE{Update the matrices $C_c=\diag\{C_c,\bar\alpha\}$,
			$S_c=\diag\{S_c,\bar\beta\}$,
			$U_c=[U_c,\bar u]$, $V_c=[V_c,\bar v]$,
			$X_c=[X_c,\bar x]$, $Y_c=[Y_c, \bar y]$,
			and set $k_c=k_c+1$.}
		
		\STATE{\textbf{if} $k_c=\ell$ \textbf{then} return $(C_c,S_c,U_c,V_c,X_c)$ and stop. \textbf{fi}}
		
		\STATE{\label{step:9}Purge $\bar x,\bar u,\bar v$
from  $\mathcal{X},\UU,\V$, update the orthonormal basis matrices $\widetilde X$, $\widetilde U$, $\widetilde V$
			and projection matrices $R_A$, $R_B$. Set $k=k-1$, and go to step \ref{step:4}.}
		\ENDIF
		
		\STATE{\label{step:11}Form $X_{p}=[X_c,\bar x]$ and $Y_{p}=[Y_c,\bar y]$.
			Iteratively solve the correction equation
			\begin{equation*}
				(I-Y_pX_p^T)(A^TA-\rho^2 B^TB)(I-X_pY_p^T)t=-(I-Y_cX_c^T)r
				\quad\mbox{for}\quad t\perp Y_p
			\end{equation*}
			with $\rho=\tau$ or $\rho=\bar \alpha/\bar\beta$ until
			the relative residual norm fulfills \eqref{inncov}.}
		
		\STATE{\label{step:12}\textbf{if} $k=k_{\max}$ \textbf{then} perform
the thick-restart, and set $k=k_{\min}$. \textbf{fi}}
		
		\STATE{\label{step:13}Orthonormalize $t$ against $\widetilde X$
			to obtain the expansion vector $x_{+}$, and set $k=k+1$.}
		\ENDWHILE
	\end{algorithmic}
\end{algorithm}

\subsection{Thick restart refined and refined harmonic JDGSVD algorithms}\label{subsec:5-3}

Algorithm~\ref{algorithm:1} summarizes the  thick-restart
RCPF-JDGSVD, RCPF-HJDGSVD and RIF-HJDGSVD algorithms with
deflation and purgation for computing a partial
GSVD $(C_c,S_c,U_c,V_c,X_c)$ of $(A,B)$ with
the $\ell$ generalized singular values
closest to a given target $\tau$.
Each of these three algorithms demands the modules to
perform matrix-vector multiplications with $A$, $B$ and
$A^T$, $B^T$, the target $\tau>0$, the number $\ell$ of
desired GSVD components, a unit-length vector $x_0$
for the initial right searching subspace, and the stopping
tolerance $\textit{tol}$ for the outer iterations.
RCPF-HJDGSVD needs a device to act
$(B^TB)^{-1}$ so as to form and update the projection matrix
$H_{\mathrm{c}}$ in \eqref{cpfeq20}.
Other parameters include the minimum and maximum
dimensions $k_{\min}$ and $k_{\max}$ of searching subspaces,
the switching tolerance $\textit{fixtol}>0$ for
the correction equations \eqref{cortau} and \eqref{deflat} with the
fixed shift $\rho=\tau$ to their counterparts with the adaptively changing
shift $\rho=\theta$, and the stopping tolerance
$\tilde\varepsilon>0$ in \eqref{inncov} for the inner iterations.
By default, we set them as $3$, $30$, $10^{-4}$ and $10^{-4}$, respectively.

The authors in \cite{huang2022harmonic} have made a detailed
analysis on the cost of the thick-restart CPF-HJGSVD and
IF-HJDGSVD algorithms, which straightforwardly applies to  Algorithm~\ref{algorithm:1}.
The conclusion is that, for $k_{\max}\ll n$, the cost of each
algorithm is dominated by the matrix-vector multiplications
with $A,B$ and $A^T,B^T$, assuming that the MINRES method is
used to solve all the correction equations.
We remind that, for the thick-restart CPF-HJGSVD and RCPF-HJGSVD, this
conclusion requires that each linear system with the
coefficient matrix $B^TB$ can be solved efficiently at
cost of $\mathcal{O}(n)$ flops.

\section{Numerical experiments}\label{sec:6}
We illustrate the
performance of thick-restart RCPF-JDGSVD, RCPF-HJDGSVD and RIF-HJDGSVD algorithms
on several problems, and compare them with the standard and harmonic
JDGSVD algorithms \cite{huang2022harmonic,huang2023cross}: CPF-JDGSVD,
CPF-HJDGSVD and IF-HJDGSVD, showing the considerable
advantage of the refined and refined harmonic algorithms over their
respective unrefined counterparts.
All the numerical experiments were implemented on a
13th Gen Intel (R) Core (TM) i9-13900KF CPU 3.00 GHz
with 64 GB RAM using the MATLAB R2022b with the
machine precision $\epsilon_{\mathrm{mach}}=2.22\times10^{-16}$ under the Miscrosoft Windows 11 Pro 64-bit system.

\begin{table}[tbhp]
	{\small\caption{Test matrix pairs with their basic properties.}\label{table0}
		\begin{center}
			\begin{tabular}{cccccccccc} \toprule
				$A$&$B$&$m$&$p$&$n$&$nnz$&$\kappa(\bsmallmatrix{A\\B})$
				&$\sigma_{\max}$&$\sigma_{\min}$ \\ \midrule
				$\mathrm{dano3mip}^T$	&$T$ &15851 &3202  &3202  &91237  &1.81e+3 &1.28e+3 &1.00e-16\\
				$\mathrm{plddb}^T$		&$T$ &5049  &3069  &3069  &20044  &91.9    &61.2    &3.65e-3\\
				$\mathrm{barth}$		&$T$ &6691  &6691  &6691  &46510  &5.56    &3.16    &7.61e-19\\
				$\mathrm{large}^T$		&$T$ &8617  &4282  &4282  &33479  &3.53e+3 &2.42e+3 &2.25e-3\\
				$\mathrm{ns3Da}$		&$T$ &20414 &20414 &20414 &1740839&5.01    &4.02e-1 &1.86e-4\\
				$\mathrm{e40r0100}$		&$T$ &17281 &17281 &17281 &605403 &11.2    &9.56    &2.62e-8\\
				$\mathrm{rat}^T$		&$T$ &9408  &3136  &3136  &278314 &3.06    &1.42    &2.85e-1\\
				\!\!\!\!$\mathrm{Kemelmacher}$\!\!\!\!	&$T$ &28452 &9693  &9693  &129952 &68.0    &2.35e+2 &2.02e-3\\
				$\mathrm{lpi\_gosh}^T$	&$T$ &13455 &3792  &3792  &111327 &3.09e+2 &1.60e+2 &3.31e-16\\
				$\mathrm{rdist2}$		&$T$ &3198  &3198  &3198  &66426  &1.17e+3 &4.50e+2 &2.64e-5\\
				$\mathrm{deter7}^T$		&$T$ &18153 &6375  &6375  &56254  &6.89    &7.97    &6.97e-3\\
				$\mathrm{shyy41}$		&$T$ &4720  &4720  &4720  &34200  &1.89e+2 &1.85e+2 &8.26e-21\\
				$\mathrm{nemeth01}$     &$D$ &9506  &9505  &9506  &744064  &32.0 &7.61e+4	&2.06e-1\\
				$\mathrm{raefsky1}$     &$D$ &3242  &3241  &3242  &299891  &6.81 &8.13e+2 &2.18e-4\\
				$\mathrm{r05}^T$       &$D$ &9690  &5189  &5190  &114523   &62.4  &1.19e+4  &2.91e-1 \\
				$\mathrm{p010}^T$       &$D$ &19090 &10089 &10090 &138178  &60.2 &1.95e+4 &2.92e-1\\
				\!\!\!\!$\mathrm{scagr7\mbox{-}2b}^T$\!\!\!\!  &$D$ &13847 &9742  &9742  &55369  &2.60e+2 &8.93e+3 &9.20e-3\\
				$\mathrm{cavity16}$     &$D$ &4562  &4561  &4562  &147009  &2.19e+2 &1.32e+3 &9.85e-7\\
				$\mathrm{utm5940}$      &$D$ &5940  &5939  &5940  &95720   &52.5 &8.91e+2 &4.10e-9\\
				\bottomrule
			\end{tabular}
	\end{center}}
\end{table}

In Table~\ref{table0} we list all the test problems and some
of their basic properties, where for each matrix pair $(A,B)$,
$A$ or $A^T$ (so that $m\geq n$) is a sparse matrix from the
University of Florida Sparse Matrix Collection \cite{davis2011university}
and $B$ is either
$$
T=\begin{bmatrix}
		3&1&&\\
		1&\ddots&\ddots&\\
&\ddots&\ddots&1\\&&1&3
	\end{bmatrix}
\qquad\mbox{or}\qquad
D=\begin{bmatrix}
		1&-1&&\\
		&\ddots&\ddots&\\&&1&-1
	\end{bmatrix}
$$
with $p=n$ or $n-1$, the latter of which is the scaled
discrete approximation of the first order derivative
operator in dimension one, $nnz$ is the total number of
nonzero elements in $A$ and $B$, $\kappa(\bsmallmatrix{A\\B})$
is the condition number of $\bsmallmatrix{A\\B}$, and
$\sigma_{\max}$ and $\sigma_{\min}$ are the largest and smallest
nontrivial generalized singular values of $(A,B)$ computed,
for experimental purpose, by the MATLAB built-in
function {\sf gsvd}.
Note that for the matrix pairs $(A,B)$ with $B=T$, all
generalized singular values are nontrivial ones while
for those with $B=D$, there exists at least one infinite
generalized singular value. We take $B=T$ in Experiments~\ref{largeexact}--\ref{tengsvd}
and $B=D$ in the remaining experiments.

For all the six algorithms, we take the outer stopping tolerance
$\textit{tol}=10^{-8}$, the initial vector
$x_0={\sf mod}(1\!:\!n,4)/\|{\sf mod}(1\!:\!n,4)\|$ with
{\sf mod} the MATLAB built-in function, and all the other
parameters by default as described in Section~\ref{subsec:5-3}.
For the correction equations \eqref{cortau} and \eqref{deflat},
we take the $n$-dimensional zero vectors as initial guesses,
and use the MATLAB built-in function {\sf minres} to solve
them until it converges with the prescribed tolerance
$\tilde \varepsilon=10^{-4}$ in \eqref{inncov} or maximum
$n$ inner iteration steps have been consumed.
We stop each algorithm and output the computed partial GSVD
of $(A,B)$ if all the $\ell$ desired GSVD components of
$(A,B)$ have been found or the total $n$ correction equations
have been solved.

In all the tables, we abbreviate the thick-restart CPF-JDGSVD,
CPF-HJDGSVD, IF-HJDGSVD and RCPF-JDGSVD, RCPF-HJDGSVD,
RIF-HJDGSVD algorithms as CPF, CPFH, IFH, and RCPF, RCPFH,
RIFH, respectively.
We denote by $I_{\mathrm{out}}$ and $I_{\mathrm{in}}$ the
total numbers of outer and inner iterations used, and by
$T_{\mathrm{cpu}}$ the total CPU time in second counted by
the MATLAB built-in commands {\sf tic} and {\sf toc}.

\begin{exper}\label{largeexact}
	We compute one GSVD component of $(A,B) = (\mathrm{dano3mip}^T, T)$
	with the target $\tau = 3.75e+2$.
	The desired  $\sigma_*$ is one of the largest
	generalized singular values of $(A,B)$ and is
	clustered with its nearby ones.
\end{exper}

\begin{figure}[tbhp]
	\centering
	\includegraphics[width=0.80\textwidth]{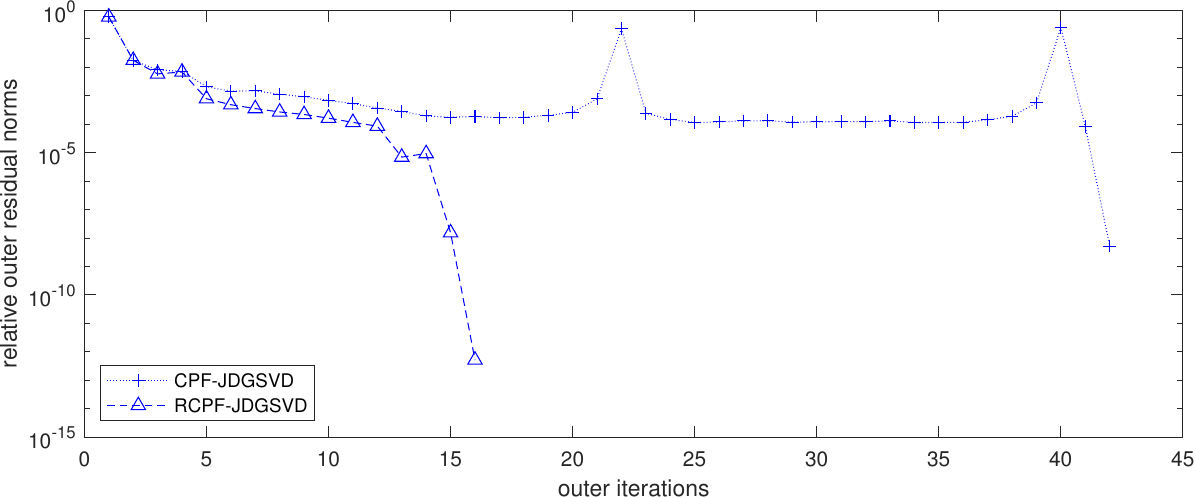}
	\caption{Computing one GSVD component of
		$(A,B)=(\mathrm{dano3mip}^T,T)$ with $\tau=3.75e+2$.}\label{fig11}
\end{figure}

For the matrix pairs in this and the next four experiments,
the matrices $B^TB$'s are symmetric positive definite and
well conditioned, whose Cholesky factorizations can be
efficiently computed using $\mathcal{O}(n)$ flops,
meaning that matrix-vector multiplications with its inversion
$(B^TB)^{-1}$ can be applied at cost of $\mathcal{O}(n)$ flops
for each linear system.
For this reason, during the subspace expansion phase of the
CPF-HJDGSVD and RCPF-HJDGSVD algorithms, we exploit the MATLAB
built-in command $\setminus$ to implement matrix-vector
multiplications with $(B^TB)^{-1}$ and to update the
intermediate matrix $H_{c}$ in \eqref{cpfeq20}.
Moreover, for experimental purpose, in this and the next
experiments we solve all the correction equations
\eqref{cortau} and \eqref{deflat} accurately by using
the LU factorizations of $(A^TA-\rho^2B^TB)$'s in order
to exhibit how RCPF-JDGSVD, RCPF-HJDGSVD and RIF-HJDGSVD
truly behave.

Figure~\ref{fig11} depicts the outer convergence curves of
CPF-JDGSVD and RCPF-JDGSVD for computing the desired GSVD
component of $(A,B)$. Clearly, CPF-JDGSVD converges slowly and irregularly and
even suffers from two sharp oscillations at outer iterations 22 and
40, respectively, and it uses 42 outer iterations to converge.
In contrast, RCPF-JDGSVD exhibits far smoother outer convergence
behavior, and uses only 16 outer iterations to attain the convergence,
thereby saving $62\%$ outer iterations.
This demonstrates the considerable superiority of
RCPF-JDGSVD to CPF-JDGSVD when computing an extreme
GSVD component of $(A,B)$.

\begin{exper}\label{interiorexact}
	We compute one GSVD component of $(A,B)=(\mathrm{plddb}^T,T)$
	with the interior generalized singular value $\sigma_*$
    closest to the target $\tau = 10$, which is clustered with its neighbors.
\end{exper}
\begin{figure}[tbhp]
	\centering
	\includegraphics[width=0.80\textwidth]{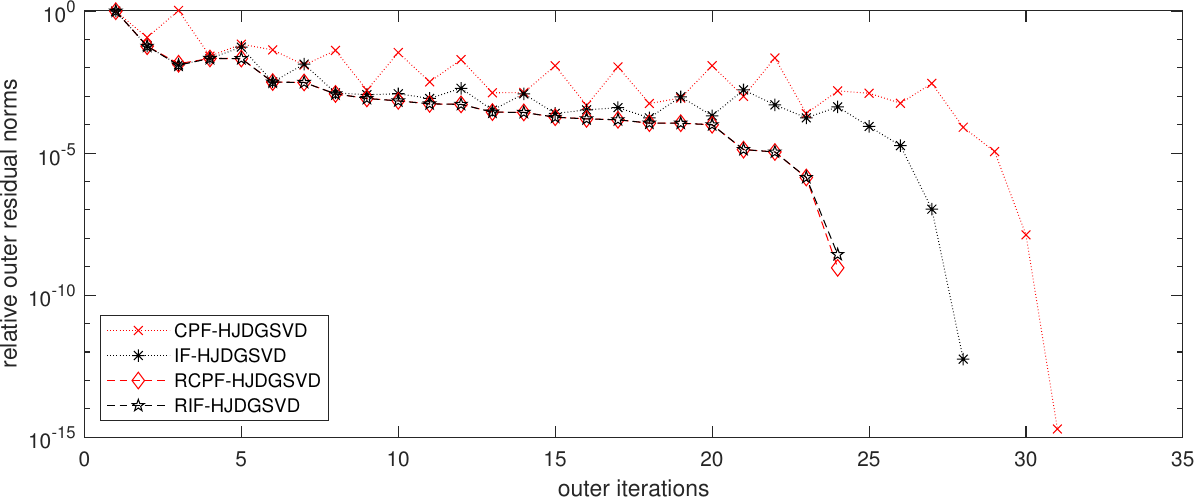}
	\caption{Computing one GSVD component of
		$(A,B)=(\mathrm{plddb}^T,T)$ with $\tau=10$.}\label{fig12}
\end{figure}

Figure~\ref{fig12} draws the outer convergence curves of the
two harmonic extraction-based algorithms CPF-HJDGSVD
and IF-HJDGSVD as well as the two refined harmonic extraction-based
algorithms RCPF-HJDGSVD and RIF-HJDGSVD.
We can see that RCPF-HJDGSVD and RIF-HJDGSVD converge more quickly
and much more smoothly than CPF-HJDGSVD and IF-HJDGSVD,
and the former two algorithms use $4$ and $7$ fewer outer iterations
than the latter two ones to reach the convergence, respectively.
We also observe that RCPF-HJDGSVD and RIF-HJDGSVD
perform equally well for this problem.

\begin{exper}\label{multiplelarge}
We compute one, five and ten largest GSVD components of
$(A,B)\\=(\mathrm{barth},T)$ with the target $\tau = 20$.
The desired generalized singular values are
clustered one another.
\end{exper}

%
%

\begin{table}[tbhp]
	\caption{Computing the $\ell$ GSVD components of $(A,B)=(\mathrm{barth},T)$
		with $\tau=20$.}\label{table21}
	\begin{center}
		\begin{tabular}{crcccrccc} \toprule
			$\ell$
			&Algorithm&$I_{\rm out}$&$I_{\rm in}$&$T_{\rm cpu}$
			&Algorithm&$I_{\rm out}$&$I_{\rm in}$&$T_{\rm cpu}$\\\midrule
			\multirow{3}{*}{1}
			&CPF\ \ \ &36 &655 &0.18
			&RCPF\ \ \ &33 &673 &0.18\\
			&CPFH\ \ \ &36 &653 &0.17
			&RCPFH\ \ \ &33 &673 &0.19\\
			&IFH\ \ \ &35 &697 &0.18
			&RIFH\ \ \ &33 &673 &0.18\\[0.25em]
			
			\multirow{3}{*}{5}
			&CPF\ \ \ &86 &2515 &0.56
			&RCPF\ \ \ &67 &2543 &0.55\\
			&CPFH\ \ \ &86 &2560 &0.58
			&RCPFH\ \ \ &67 &2543 &0.59\\
			&IFH\ \ \ &81 &2609 &0.56
			&RIFH\ \ \ &61 &2466 &0.51\\[0.25em]
			
			\multirow{3}{*}{10}
			&CPF\ \ \ &166 &5575 &1.25
			&RCPF\ \ \ &91 &4834 &0.94\\
			&CPFH\ \ \ &173 &5634 &1.30
			&RCPFH\ \ \ &91 &4834 &0.97\\
			&IFH\ \ \ &155 &5603 &1.25
			&RIFH\ \ \ &109 &5269 &1.09
			\\\bottomrule
		\end{tabular}
	\end{center}
\end{table}

Table~\ref{table21} reports the results obtained.
As we can see, for $\ell=1$, the refined and refined
harmonic extraction-based JDGSVD algorithms use
slightly fewer outer iterations and comparable inner
iterations and CPU time to achieve the convergence,
compared with the standard and harmonic extraction-based
JDGSVD algorithms.
However, for $\ell=5$, RCPF-JDGSVD, RCPF-HJDGSVD and
RIF-HJDGSVD are considerably faster and use nearly
$20$ fewer outer iterations than CPF-JDGSVD, CPF-HJDGSVD
and IF-HJDGSVD, respectively, which save
$22.1\%\sim 24.5\%$ of the outer iterations, though the six JDGSVD
algorithms use almost the same inner iterations and
CPU time to converge.
For $\ell=10$, RIF-HJDGSVD uses $29\%$ fewer
outer iterations, $6\%$ fewer inner iterations and $13\%$
less CPU time than IF-HJDGSVD; RCPF-JDGSVD and RCFP-HJDGSVD
save more than $45\%$ of the outer iterations, $13\%$ of the inner
iterations, and $24\%$ of the CPU time, respectively, relative to
CPF-JDGSVD and CPF-HJDGSVD.
Of the three refined and refined harmonic
algorithms, RCPF-JDGSVD and RCPF-HJDGSVD are better than
RIF-HJDGSVD due to their faster outer convergence,
and RCPF-JDGSVD is the most efficient due to its least CPU time.

Clearly, for the sake of faster outer convergence and
higher overall efficiency, RCPF-JDGSVD is the most
recommended algorithm for this problem.
This should be expected as the generalized singular values
of interest are extreme ones, for which the standard
extraction suits better than the harmonic extraction.
As a consequence, RCPF-JDGSVD performs better than
RCPF-HJDGSVD and RIF-HJDGSVD, though its advantage
is not so obvious relative to RCPF-HJDGSVD.

\begin{exper}\label{multipleinterior}
We compute one, five and ten interior GSVD components of
the matrix pair $(A,B)=(\mathrm{large}^T,T)$ with
the generalized singular values closest to $\tau = 14$ highly
clustered.
\end{exper}

\begin{figure}[tbhp]
	\centering
	\includegraphics[width=0.80\textwidth]{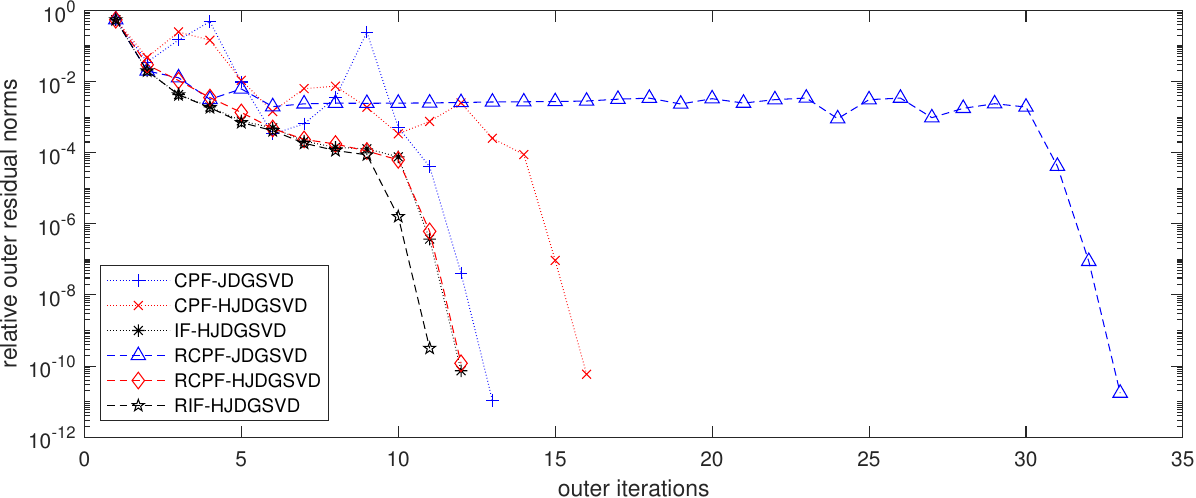}
	\caption{Computing one GSVD component of
		$(A,B)=(\mathrm{large}^T,T)$ with $\tau=14$.}\label{fig22}
\end{figure}

%
%

\begin{table}[tbhp]
	\caption{Computing the $\ell$ GSVD components of $(A,B)=(\mathrm{large}^T,T)$
		with $\tau=14$.}\label{table22}
	\begin{center}
		\begin{tabular}{crcccrccc} \toprule
			$\ell$
			&Algorithm&$I_{\rm out}$&$I_{\rm in}$&$T_{\rm cpu}$
			&Algorithm&$I_{\rm out}$&$I_{\rm in}$&$T_{\rm cpu}$\\ \midrule
			\multirow{3}{*}{1}
			&CPF\ \ \ &13&3266&0.23
			&RCPF\ \ \ &33&6068&0.44\\
			&CPFH\ \ \ &16&2677&0.18
			&RCPFH\ \ \ &12&3109&0.22\\
			&IFH\ \ \ &12&3544&0.24
			&RIFH\ \ \ &11&3332&0.22\\[0.25em]
			
			\multirow{3}{*}{5}
			&CPF\ \ \ &65&42115&2.84
			&RCPF\ \ \ &720&635121&43.5\\
			&CPFH\ \ \ &43&13255&0.93
			&RCPFH\ \ \ &25&16434&1.11\\
			&IFH\ \ \ &29&15602&1.06
			&RIFH\ \ \ &27&23767&1.61\\[0.25em]
			
			\multirow{3}{*}{10}
			&CPF\ \ \!&4291&4084862&358
			&RCPF\ \ \!&771&720147&50.7\\
			&CPFH\ \ \!&61&29032&2.27
			&RCPFH\ \ \!&39&33313&2.48\\
			&IFH\ \ \!&43&30135&2.23
			&RIFH\ \ \!&36&30570&2.26
			
			\\\bottomrule
		\end{tabular}
	\end{center}
\end{table}

Figure~\ref{fig22} depicts the outer convergence curves of the
six JDGSVD algorithms for computing one desired GSVD component
of $(A,B)$, and Table~\ref{table22} displays all the results
obtained.

For $\ell=1$, all the six algorithms succeed to compute the
desired GSVD component. As is seen from Figure~\ref{fig22}
and Table~\ref{table22},
RCPF-JDGSVD uses much more outer iterations and thus much
more inner iterations and CPU time to converge than CPF-JDGSVD because
of the outer convergence stagnation of the former.
Notice that this is an interior GSVD problem. The stagnation
implies that the expanded subspaces have little improvements
at those stagnation steps, causing that the accuracy of selected
approximate generalized singular values and right generalized
singular vectors remain almost the same.
This is because RCPF-JDGSVD selects approximate generalized singular
values incorrectly or there is no good Ritz value at all
at those steps, which is an intrinsic
deficiency of the standard extraction for interior
eigenvalue, SVD and GSVD problems, as we have pointed out
in the introduction:
If this selection is done incorrectly or there are spurious
values at some step, the searching subspaces in the refined extraction are
expanded wrongly at that step and provide little
information on the desired generalized singular vectors.
In contrast, RCPF-HJDGSVD and RIF-HJDGSVD outperform
CPF-HJDGSVD and IF-HJDGSVD respectively, due to the
smoother and faster outer convergence and/or
better overall efficiency.
This is because the harmonic extraction approach
suits better for interior GSVD problems and can pick up the
approximate generalized singular values correctly.

For $\ell=5$, it seems from Table~\ref{table22} that
RCPF-JDGSVD uses significantly more outer and inner iterations
and much more CPU time than CPF-JDGSVD. As a matter of fact,
both CPF-JDGSVD and RCPF-JDGSVD are unreliable and fail
to deliver all the desired GSVD components in this case:
They find  only the first three desired GSVD components
but recompute each of the first two twice.
RCPF-HJDGSVD uses nearly half of the outer iterations though a
little bit more inner iterations and CPU time relative
to CPF-HJDGSVD. RIF-HJDGSVD uses slightly fewer and
considerably more inner iterations and CPU time than
IF-HJDGSVD. Obviously, regarding outer convergence,
RCPF-HJDGSVD and RIF-HJDGSVD are the best.

For $\ell=10$, only CPF-HJDGSVD, IF-HJDGSVD and RCPF-HJDGSVD,
RIF-HJDGSVD succeed to compute all the desired GSVD components
of $(A,B)$, while CPF-JDGSVD computes the first three desired
ones three times in the beginning and fails to compute the
fourth one when $n$ correction equations are solved, and
RCPF-JDGSVD only computes the first six desired GSVD components
and recomputes the first three ones twice or three times.
As a consequence, CPF-JDGSVD and RCPF-JDGSVD fail to
solve the concerning GSVD problem correctly and are thus
unreliable for this interior GSVD problem. On the contrary,
RCPF-HJDGSVD and RIF-HJDGSVD outperform CPF-HJDGSVD and
IF-HJDGSVD as they use considerably fewer outer iterations
and comparable inner iterations; these four algorithms
outperform CPF-JDGSVD and RCPF-JDGSVD significantly.

Clearly, due to the smoother and faster outer convergence and the
reliability,
RCPF-HJDGSVD and RIF-HJDGSVD are the most suitable choices
for this interior GSVD problem.

\begin{exper}\label{tengsvd}
We compute ten GSVD components of the matrix pairs
$(A,B)=(\mathrm{ns3Da},T)$, $(\mathrm{e40r0100},T)$,
$(\mathrm{Kemelmacher},T)$, $(\mathrm{rat}^T,T)$, $(\mathrm{lpi\_gosh}^T,T)$, $(\mathrm{rdist2},T)$, $(\mathrm{deter7}^T,T)$ and
$(\mathrm{shyy41},T)$ with the targets $\tau = 2$, $11$,
$0.2$, $0.15$, $14.3$, $19$, $4.5$ and $0.8$, respectively.
The desired generalized singular values of $(\mathrm{ns3Da},T)$,
$(\mathrm{e40r0100},T)$ and $(\mathrm{Kemelmacher},T)$,
$(\mathrm{rat}^T,T)$ are extreme, i.e., the largest or smallest,
and are clustered, and
those of the remaining four matrix pairs are both interior and clustered.
\end{exper}	

\begin{table}[tbhp]
	\caption{Results on test matrix pairs in Example~\ref{tengsvd}, where $\mathrm{Kemelm}$ is the abbreviation of $\mathrm{Kemelmacher}$.}\label{table3}
	\begin{center}
		\begin{tabular}{crcccrccc} \toprule
			$A$
			&\!\!Algorithm\!\!&\!$I_{\rm out}$\!&$I_{\rm in}$&$T_{\rm cpu}$
			&\!\!Algorithm\!\!&\!$I_{\rm out}$\!&$I_{\rm in}$&$T_{\rm cpu}$\\\midrule
			
			\multirow{3}{*}{$\mathrm{ns3Da}$}
			&CPF\ \ \ &126&3823&6.81
			&RCPF\ \ \ &89&3458&6.03 \\
			&CPFH\ \ \ &119&3858&6.87
			&RCPFH\ \ \ &90&3471&6.17 \\
			&IFH\ \ \ &113&3776&6.85
			&RIFH\ \ \ &100&3646&6.49 \\[0.25em]
			
			\multirow{3}{*}{$\mathrm{e40r0100}$}
			&CPF\ \ \ &141&9507&5.20
			&RCPF\ \ \ &58&7656&3.86\\
			&CPFH\ \ \ &91&9094&4.74
			&RCPFH\ \ \ &58&7549&3.87\\
			&IFH\ \ \ &84&9020&4.74
			&RIFH\ \ \ &58&7615&3.83\\ [0.25em]
			
			\multirow{3}{*}{\!$\mathrm{Kemelm}$\!}
			&CPF\ \ \ &45&48004&9.82
			&RCPF\ \ \ &34&45001&9.09\\
			&CPFH\ \ \ &39&48669&10.1
			&RCPFH\ \ \ &37&48271&10.0\\
			&IFH\ \ \ &35&45056&9.25
			&RIFH\ \ \ &36&46594&9.60\\ [0.25em]
			
			\multirow{3}{*}{$\mathrm{rat}^T$}
			&CPF\ \ \ &172&2364&0.70
			&RCPF\ \ \ &81&2289&0.54 \\
			&CPFH\ \ \ &94&2191&0.54
			&RCPFH\ \ \ &80&2232&0.56 \\
			&IFH\ \ \ &105&2118&0.54
			&RIFH\ \ \ &81&2289&0.54 \\ [0.25em]
			
			\multirow{3}{*}{$\mathrm{lpi\_gosh}^T$}
			&CPF\ \ \ &52&6042&0.81
			&RCPF\ \ \ &43&6255&0.77\\
			&CPFH\ \ \ &44&5661&0.71
			&RCPFH\ \ \ &37&5203&0.63\\
			&IFH\ \ \ &39&5405&0.67
			&RIFH\ \ \ &36&5023&0.60\\ [0.25em]

			\multirow{3}{*}{$\mathrm{rdist2}$}
			&CPF\ \ \ &68&5397&0.48
			&RCPF\ \ \ &46&4476&0.39\\
			&CPFH\ \ \ &51&4485&0.40
			&RCPFH\ \ \ &44&4348&0.40\\
			&IFH\ \ \ &54&5238&0.45
			&RIFH\ \ \ &47&4734&0.41\\ [0.25em]
			
			\multirow{3}{*}{$\mathrm{deter7}^T$}
			&CPF\ \ \ &87&7208&1.12
			&RCPF\ \ \ &68&6820&1.04\\
			&CPFH\ \ \ &78&6876&1.07
			&RCPFH\ \ \ &55&6716&0.97\\
			&IFH\ \ \ &69&7430&1.11
			&RIFH\ \ \ &57&7418&1.06\\[0.25em] 			
			
			\multirow{3}{*}{$\mathrm{shyy41}$}
			&CPF\ \ \ &\!\!4727\!\!&\!7644960\!&\!\!639\!\!
			&RCPF\ \ \ &\!\!4724\!\!&\!3555937\!&263\\
			&CPFH\ \ \ &\!\!4727\!\!&\!4493658\!&\!\!1.13e+3\!\!
			&RCPFH\ \ \ &66&81061&6.35\\
			&IFH\ \ \ &\!\!4728\!\!&\!\!17581871\!\!&\!\!1.49e+3\!\!
			&RIFH\ \ \ &63&79204&6.29 \\
			\bottomrule
		\end{tabular}
	\end{center}
\end{table}

Table~\ref{table3} reports all the results obtained.
For the computation of the largest GSVD components of
$(\mathrm{ns3Da},T)$ and $(\mathrm{e40r0100},T)$,
we observe from Table~\ref{table3} that the
refined and refined harmonic extraction-based RCPF-JDGSVD, RCPF-HJDGSVD, RIF-HJDGSVD algorithms outperform
the standard and harmonic extraction-based CPF-JDGSVD,
CPF-HJDGSVD, IF-HJDGSVD algorithms as they use
moderately to substantially fewer outer iterations,
slightly to moderately fewer inner iterations and less CPU time.
For $(\mathrm{ns3Da},T)$, RCPF-JDGSVD and RCPF-HJDGSVD are very
comparable, and they both are superior to RIF-HJDGSVD due to
the faster outer convergence and higher overall efficiency.
For $(\mathrm{e40r0100},T)$, all the three refined and refined
harmonic JDGSVD algorithms are suitable.

For the computation of the smallest GSVD components, we see
from Table~\ref{table3} that for $(\mathrm{Kemelmacher},T)$,
RCPF-JDGSVD and RCPF-HJDGSVD use moderately or slightly fewer
outer and inner iterations and less CPU time than CPF-JDGSVD
and CPF-HJDGSVD, respectively, while RIF-HJDGSVD and IF-HJDGSVD
work equally well.
Obviously, except for CPF-JDGSVD, all the other five
algorithms are appropriate for the concerning GSVD problem
of this matrix pair.
Nonetheless, as far as both the outer convergence and
overall efficiency are concerned, RCPF-JDGSVD is the
most recommended.
For $(\mathrm{rat}^T,T)$, the refined and refined harmonic
extraction-based JDGSVD algorithms outperform the standard and harmonic extraction-based JDGSVD algorithms substantially,
as outer iterations indicate.
Regarding the overall efficiency, RCPF-JDGSVD outmatches
CPF-JDGSVD considerably as it uses slightly fewer
inner iterations and much less CPU time.
On the other hand, RCPF-HJDGSVD and RIF-HJDGSVD are
competitive with CPF-HJDGSVD and IF-HJDGSVD in terms of inner
iterations and CPU time.
For the sake of better convergence behavior, RCPF-JDGSVD,
RCPF-HJDGSVD and RIF-HJDGSVD are all proper choices
for this problem.

For the computation of interior GSVD components of
$(\mathrm{lpi\_gosh}^T,T)$, $(\mathrm{rdist2},T)$ and
$(\mathrm{deter7}^T,T)$, as far as the outer convergence
is concerned, RCPF-JDGSVD, RCPF-HJDGSVD and RIF-HJDGSVD
surpass CPF-JDGSVD, CPF-HJDGSVD and IF-HJDGSVD, respectively.
Regarding the overall efficiency, the three refined and refined
harmonic JDGSVD algorithms are superior to the standard and
harmonic JDGSVD algorithms as they use less CPU time and/or fewer inner iterations.

For $(\mathrm{shyy41},T)$, since the desired generalized
singular values are truly interior and clustered, as we
have elaborated in the introduction, the standard, harmonic
and even refined extraction approach-based JDGSVD algorithms
may face severe difficulties, which is confirmed by the
results in Table~\ref{table3}.
In fact, we observe very erratic convergence
behavior of CPF-JDGSVD, CPF-HJDGSVD, IF-HJDGSVD and
RCPF-JDGSVD; when $n$ correction equations are solved,
they fail to compute all the desired GSVD components and
only output seven, seven, eight and four converged ones, respectively.
On the contrary, RCPF-HJDGSVD and RIF-HJDGSVD are very successful
to compute all the desired GSVD components quickly.
Undoubtedly, RCPF-HJDGSVD and RIF-HJDGSVD are the only two
proper choices for this difficult problem; they are
competitive in terms of both outer convergence and overall efficiency.

\begin{exper}\label{6}
	 We compute one GSVD component of $(A,B)=(\mathrm{nemeth01},D)$
corresponding to the interior generalized singular value closest to the target $\tau=6.5$.
\end{exper}

\begin{figure}[tbhp]
	\centering
	\includegraphics[width=0.80\textwidth]{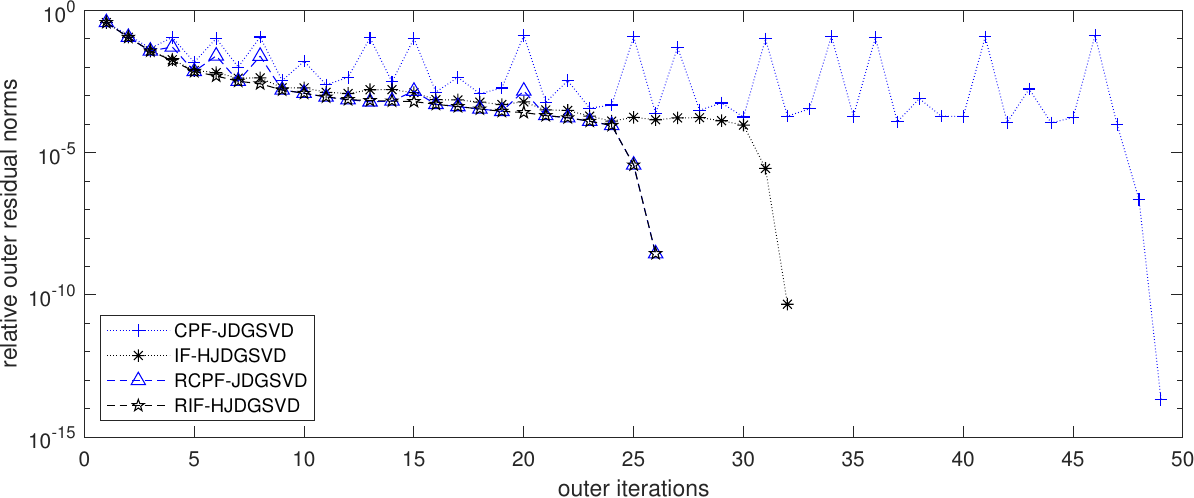}
	\caption{Computing one GSVD component of
		$(A,B)=(\mathrm{nemeth01},D)$ with $\tau=6.5$.}\label{fig4}
\end{figure}

For the matrix pairs in this and the next experiments, the
matrices $B$'s are rank deficient, so that CPF-HJDGSVD and
RCPF-HJDGSVD are not applicable.
We compute the desired GSVD components of $(A,B)$ using
CPF-JDGSVD, IF-HJDGSVD and RCPF-JDGSVD, RIF-HJDGSVD.
Specifically, for the problem in this experiment, for the
illustration of true outer convergence of these four
algorithms, we use the LU factorizations of the matrices
$(A^TA-\rho^2B^TB)$'s to solve all the correction equations
involved accurately, and draw their outer convergence curves
in Figure~\ref{fig4}.

As is seen from Figure~\ref{fig4}, in contrast to CPF-JDGSVD
and IF-HJDGSVD whose convergence is delayed because of the
frequent oscillations and stagnation, respectively, RCPF-JDGSVD and
RIF-HJDGSVD converge faster and much more smoothly, and use
$23$ and $6$ fewer outer iterations to achieve the convergence,
respectively. Moreover, since the standard extraction approach may pick up
approximate generalized singular values incorrectly for interior
GSVD problems, RCPF-JDGSVD may suffer from oscillations, as shown
by Figure~\ref{fig4} at iterations $4$--$9$ and
$19$--$21$. These phenomena confirm the intrinsic shortcoming of RCPF-JDGSVD for
computing truly interior GSVD components.
As a consequence, although RCPF-JDGSVD and RIF-HJDGSVD use
the same number of outer iterations to converge and
outperform CPF-JDGSVD and IF-HJDGSVD, we recommend RIF-HJDGSVD
for this problem.

\begin{exper}\label{Brankdeficient}
	We compute ten GSVD components of
$(A,B)=(\mathrm{raefsky1},D)$, $(\mathrm{r05}^T,D)$, $(\mathrm{p010}^T,D)$ and
$(\mathrm{scagr7\mbox{-}2b}^T,D)$, $(\mathrm{cavity16},D)$,  $(\mathrm{utm5940},D)$
with the targets $\tau=57$, $61$, $80$ and
$32.5$, $20.8$, $58$, respectively. All the desired generalized
singular values are interior and clustered ones.
This implies that all the correction equations may be hard
to solve by the MINRES method.
\end{exper}

%
%
%
%
%
%

\begin{table}[tbhp]
	\caption{Results on test matrix pairs in Example~\ref{Brankdeficient}.}\label{table5}
	\begin{center}
		\begin{tabular}{crcccrccc} \toprule
			$A$
			&\!\!\!Algorithm\!\!\!&$I_{\rm out}$\!&$I_{\rm in}$&$T_{\rm cpu}$
			&\!\!\!Algorithm\!\!\!&$I_{\rm out}$\!&$I_{\rm in}$&$T_{\rm cpu}$\\\midrule
			
			\multirow{2}{*}{$\mathrm{raefsky1}$}
			&CPF\ \ \  &52  &91560&16.0
			&RCPF\ \  &42  &76689&13.4  \\
			&IFH\ \ \  &37  &67609&13.5
			&RIFH\ \  &35  &66005&11.9  \\[0.25em]
			
			\multirow{2}{*}{$\mathrm{r05^T}$}
			&CPF\ \ \  &100  &60409&8.71
			&RCPF\ \  &99  &65947&9.76  \\
			&IFH\ \ \  &91 &64795&19.2
			&RIFH\ \  &69  &53142&7.73  \\[0.25em]
			
			\multirow{2}{*}{$\mathrm{p010}^T$}
			&CPF\ \ \  &104  &104797& 28.8
			&RCPF\ \  &107  &113825& 31.4  \\
			&IFH\ \ \  &87  &132654&55.2
			&RIFH\ \  &63  &81573&22.4  \\[0.25em]
			
			\multirow{2}{*}{\!\!\!\!$\mathrm{scagr7\mbox{-}2b}^T$\!\!\!\!}
			&CPF\ \ \  &185 &130938&22.1
			&RCPF\ \  &294  &232935&39.0  \\
			&IFH\ \ \  &65  &141444&23.3
			&RIFH\ \  &46  &112561&18.4  \\[0.25em]
	
			\multirow{2}{*}{$\mathrm{cavity16}$}
			&CPF\ \ \  &122 &95455&13.5
			&RCPF\ \  &\!\!4565\!&2107949&292  \\
			&IFH\ \ \  &101 &85234& 23.7
			&RIFH\ \  &68  &63650&8.95  \\[0.25em]
			
			\multirow{2}{*}{$\mathrm{utm5940}$}
			&CPF\ \ \  &\!\!4255\!\!  &\!\!14983907\!\!&\!\!2.13e+3\!\!
			&RCPF\ \  &\!\!5949\!\!&\!\!10827153\!\!&\!\!1.55e+3\!\!  \\
			&IFH\ \ \  &\!\!1951\!&\!\!8178349\!\!&\!\!2.36e+3\!\!
			&RIFH\ \  &87  &\!\!266135\!\!&35.8  \\[0.25em]
			\bottomrule
		\end{tabular}
	\end{center}
\end{table}

Table~\ref{table5} lists the results.
For $(\mathrm{raefsky1},D)$, we see that RCPF-JDGSVD and
RIF-HJDGSVD respectively outperform CPF-JDGSVD and
IF-HJDGSVD by using significantly or slightly fewer
outer and inner iterations and fairly less CPU time.

For $(\mathrm{r05}^T,D)$ and $(\mathrm{p010}^T,D)$,
RIF-HJDGSVD uses substaintially fewer outer and inner
iterations and considerably less CPU time than
CPF-JDGSVD and RCPF-JDGSVD, which themselves behave
similarly and outperform IF-HJDGSVD significantly with
much less CPU time.
Obviously, RIF-HJDGSVD is the best choice for these two
problems in terms of both outer convergence and overall efficiency.

For $(\mathrm{scagr7\mbox{-}2b}^T,D)$, CPF-JDGSVD and
RCPF-JDGSVD consume lots of outer iterations to converge. Actually,
RCPF-JDGSVD fails to solve the problem,
and it repeatedly computes the first two
desired GSVD components.
In contrast, IF-HJDGSVD and RIF-HJDGSVD use substantially
fewer outer iterations to compute all the desired GSVD components.
Between them, RIF-HJDGSVD wins IF-HJDGSVD with obviously
fewer outer and inner iterations and  less CPU time.
Clearly, for this problem, RIF-HJDGSVD is the best choice.

For $(\mathrm{cavity16},D)$ and $(\mathrm{utm5940},D)$,
the desired generalized singular values are truly interior
and clustered, causing that the standard, harmonic and even
refined extraction approaches may not perform well, as is
confirmed and shown in Table~\ref{table5}.
We observe sharp oscillations in the outer convergence curves
of CPF-JDGSVD, IF-HJDGSVD and RCPF-JDGSVD, of which the last
one even fails to compute all the desired GSVD components when
we use up the solutions of $n$ correction equations.
On the other hand, RIF-HJDGSVD converges smoothly and uses
marvelously fewer outer and inner iterations and extremely
less CPU time than the other three algorithms.
Clearly, RIF-HJDGSVD is the most suitable choice for these two problems.

\vspace{0.6em}

Summarizing all the numerical experiments, we can draw
two conclusions:
(\romannumeral1) For extreme GSVD components, refined and
refined harmonic RCPF-JDGSVD, RCPF-HJDGSVD and RIF-HJDGSVD
algorithms generally perform better than the standard and
harmonic CPF-JDGSVD, CPF-HJDGSVD and IF-HJDGSVD algorithms,
respectively; RCPF-JDGSVD is the best.
(\romannumeral2) For interior GSVD components, RCPF-HJDGSVD
and RIF-HJDGSVD outmatch CPF-JDGSVD, CPF-HJDGSVD and
IF-HJDGSVD,  RCPF-JDGSVD with smoother and faster outer
convergence and higher overall efficiency;
if $B$ has full column rank, then both RCPF-HJDGSVD and
RIF-HJDGSVD are suitable choices; otherwise RIF-HJDGSVD
is the most recommended algorithm due to its wider applicability.

\section{Conclusions}\label{sec:7}

The reliable and efficient computation of a partial GSVD
of a large regular matrix pair $(A,B)$ is vital in extensive
applications, and has attracted much attention in recent years.
In this paper, we have proposed three refined and refined
harmonic extraction-based JDGSVD methods: RCPF-JDGSVD, and
RCPF-HJDGSVD, RIF-HJDGSVD, and have developed practical
thick-restart algorithms with deflation and purgation that can
compute several GSVD components.
They fix erratic convergence behavior and intrinsic possible
non-convergence of the standard and harmonic JDGSVD methods:
CPF-JDGSVD, CPF-HJDGSVD and IF-HJDGSVD proposed and developed
in \cite{huang2022harmonic,huang2023cross}.
The three new JDGSVD algorithms suit better for the computation
of extreme and interior GSVD components of large regular matrix
pairs, respectively. Numerical experiments have demonstrated that
RCPF-HJDGSVD and RIF-HJDGSVD
are generally the best choices for computing interior GSVD
components, while RCPF-JDGSVD suits best for the extreme GSVD components.
These confirm our elaborations in the introduction
and the necessity and superiority of refined and refined harmonic
extraction-based JDGSVD methods.

There remain some important issues that should be given special
considerations.
When the desired generalized singular values are truly interior
and clustered, the correction equations \eqref{cortau} and \eqref{deflat}
involved in JDGSVD
algorithms are highly indefinite and ill conditioned, causing
that the MINRES method may be very costly even if the relative
residual of an approximate solution is only required to be fairly small.
Unfortunately, we have observed that commonly used incomplete
LU preconditioners generally work poorly and have no
acceleration effect.
Therefore, the efficient solutions of correction equations
constitute the bottleneck of all JDGSVD algorithms.
How to propose and develop specific preconditioners for the
correction equations is extremely important and definitely
deserves enough attention.
This constitutes our future work.
\bigskip

{\bf Funding} The first author was supported by the Youth Fund
of the National Science Foundation of China (No. 12301485) and
the Youth Program of the Natural Science Foundation of Jiangsu
Province (No. BK20220482), and the second author was supported
by the National Science Foundation of China (No.12171273).
\bigskip

{\bf Data Availability}
Enquires about data availability should be directed to the authors.

\section*{Declarations}

The two authors declare that they have no
financial interests, and the two authors read and approved the final
manuscript.

\bibliographystyle{siamplain}

\end{document}